\documentclass[11pt,fleqn]{article}
\usepackage{graphicx}
\usepackage{amsthm,amsmath}
\usepackage{amssymb,dsfont,mathrsfs}

\numberwithin{equation}{section}
\newtheorem{theorem}{Theorem}[section]
\newtheorem{lemma}{Lemma}[section]
\newtheorem{proposition}{Proposition}[section]
\newtheorem{corollary}{Corollary}[section]
\newtheorem{remark}{Remark}[section]
\newtheorem{assumption}{Assumption}[section]

\def\ba{\boldsymbol{a}}
\def\bb{\boldsymbol{b}}
\def\bc{\boldsymbol{c}}

\def\bg{\boldsymbol{g}}

\def\bl{\boldsymbol{l}}

\def\bp{\boldsymbol{p}}

\def\bu{\boldsymbol{u}}
\def\bv{\boldsymbol{v}}

\def\bx{\boldsymbol{x}}

\def\bJ{\boldsymbol{J}}
\def\bM{\boldsymbol{M}}
\def\bX{\boldsymbol{X}}
\def\bY{\boldsymbol{Y}}

\def\bvarphi{\boldsymbol{\varphi}}

\def\bnu{\boldsymbol{\nu}}
\def\btheta{\boldsymbol{\theta}}

\def\bzero{\mathbf{0}}
\def\bone{\mathbf{1}}

\def\scrI{\mathscr{I}}

\def\spr{\mbox{\rm spr}}
\def\cp{\mbox{\rm cp}}
\def\diag{\mbox{\rm diag}}

\def\vec{\mbox{\rm vec}}
\title{Tail Asymptotics in any direction of the stationary distribution in a two-dimensional discrete-time QBD process}
\author{Toshihisa Ozawa  \\ 
Faculty of Business Administration, Komazawa University \\
1-23-1 Komazawa, Setagaya-ku, Tokyo 154-8525, Japan \\
E-mail: toshi@komazawa-u.ac.jp
}
\date{} 

\begin{document}

\maketitle

\begin{abstract}
We consider a discrete-time two-dimensional quasi-birth-and-death process (2d-QBD process for short) $\{(\boldsymbol{X}_n,J_n)\}$ on $\mathbb{Z}_+^2\times S_0$, where $\boldsymbol{X}_n=(X_{1,n},X_{2,n})$ is the level state, $J_n$ the phase state (background state) and $S_0$ a finite set, and study asymptotic properties of the stationary tail distribution. The 2d-QBD process is an extension of usual one-dimensional QBD process. 
By using the matrix analytic method of the queueing theory and the complex analytic method, we obtain the asymptotic decay rate of the stationary tail distribution in any direction. This result is an extension of the corresponding result for a certain two-dimensional reflecting random work without background processes, obtained by using the large deviation techniques. 
We also present a condition ensuring the sequence of the stationary probabilities geometrically decays without power terms, asymptotically. 
Asymptotic properties of the stationary tail distribution in the coordinate directions in a 2d-QBD process have already been studied in the literature. The results of this paper are also important complements to those results. 

\smallskip
{\it Keywards}: quasi-birth-and-death process, Markov modulated reflecting random walk, Markov additive process, asymptotic decay rate, stationary distribution, matrix analytic method

\end{abstract}

%
%
\section{Introduction} \label{sec:intro}

We deal with a two-dimensional discrete-time quasi-birth-and-death process (2d-QBD process for short), which is an extension of ordinary one dimensional QBD process (see, for example, Latouche and Ramaswami \cite{Latouche99}),  and study asymptotic properties of the stationary tail distribution \textit{in any direction}. 
The 2d-QBD process is also a two-dimensional skip-free Markov modulated reflecting random walk (2d-MMRRW for short), and the 2d-MMRRW is a two-dimensional skip-free reflecting random walk (2d-RRW for short) \textit{having a background process}.  Asymptotics of the stationary distributions in various 2d-RRWs \textit{without background processes} have been investigated in the literature for several decades, especially, by Masakiyo Miyazawa and his colleagues (see a survey paper \cite{Miyazawa11} of Miyazawa and references therein). Some of their results have been extended to the 2d-QBD process in Ozawa \cite{Ozawa13}, Miyazawa \cite{Miyazawa15} and Ozawa and Kobayashi \cite{Ozawa18}, where the asymptotic decay rates and exact asymptotic formulae of the stationary distribution \textit{in the coordinate directions} were obtained  (cf.\ results in Miyazawa \cite{Miyazawa09} and Kobayashi and Miyazawa \cite{Kobayashi13}). In this paper, we further extend it to an arbitrary direction. In Miyazawa \cite{Miyazawa15}, the tail decay rates of the marginal stationary distribution in an arbitrary direction have also been obtained. 

Let a Markov chain $\{\bY_n\}=\{(\bX_n,J_n)\}$ be a 2d-QBD process on the state space $\mathbb{Z}_+^2\times S_0$, where $\bX_n=(X_{1,n},X_{2,n})$, $S_0$ is a finite set with cardinality $s_0$, i.e., $S_0=\{1,2,...,s_0\}$, and $\mathbb{Z}_+$ is the set of all nonnegative integers.  The process $\{\bX_n\}$ is called the level process, $\{J_n\}$ the phase process (background process), and the transition probabilities of the level process vary according to the state of the phase process. This modulation is space homogeneous except for the boundaries of $\mathbb{Z}_+^2$. The level process is assumed to be skip free, i.e., for any $n\ge 0$, $\bX_{n+1}-\bX_n\in\{-1,0,1\}^2$. 
Stochastic models arising from various Markovian two-queue models and two-node queueing networks such as two-queue polling models and generalized two-node Jackson networks with Markovian arrival processes and phase-type service processes can be represented as two-dimensional continuous-time QBD processes, and their stationary distributions can be analyzed through the corresponding two-dimensional discrete-time QBD processes obtained by the uniformization technique; See, for example, Refs.\ \cite{Miyazawa15,Ozawa13,Ozawa19}. In that sense, 2d-QBD processes are more versatile than 2d-RRWs, which have no background processes. This is a reason why we are interested in stochastic models with a background process. Here we emphasize that the assumption of skip-free is not so restricted since any 2d-MMRRW \textit{with bounded jumps} can be represented as a 2d-MMRRW \textit{with skip-free jumps} (i.e., 2d-QBD process); See Introduction of Ozawa \cite{Ozawa21}. 

%
Denote by $\scrI_2$ the set of all the subsets of $\{1,2\}$, i.e., $\scrI_2=\{\emptyset,\{1\},\{2\},\{1,2\}\}$, and we use it as an index set. Divide $\mathbb{Z}_+^2$ into $2^2=4$ exclusive subsets defined as 
\[
\mathbb{B}^\alpha=\{\bx=(x_1,x_2)\in\mathbb{Z}_+^2; \mbox{$x_i>0$ for $i\in\alpha$, $x_i=0$ for $i\in\{1,2\}\setminus \alpha$} \},\ \alpha\in\scrI_2.  
\]
The class $\{\mathbb{B}^\alpha; \alpha\in\scrI_2\}$ is a partition of $\mathbb{Z}_+^2$. $\mathbb{B}^\emptyset$ is the set containing only the origin, and $\mathbb{B}^{\{1,2\}}$ is the set of all positive points in $\mathbb{Z}_+^2$. 
Let $P$ be the transition probability matrix of the 2d-QBD process $\{\bY_n\}$ and represent it in block form as $P=\left( P_{\bx,\bx'}; \bx,\bx'\in\mathbb{Z}_+^2 \right)$, where $P_{\bx,\bx'}=(p_{(\bx,j),(\bx',j')}; j,j'\in S_0)$ and $p_{(\bx,j),(\bx',j')}=\mathbb{P}(\bY_1=(\bx',j')\,|\,\bY_0=(\bx,j))$. For $\alpha\in\scrI_2$ and $i_1,i_2\in\{-1,0,1\}$, let $A^\alpha_{i_1,i_2}$ be a one-step transition probability block from a state in $\mathbb{B}^\alpha$, which is defined as 
\[
[A^\alpha_{i_1,i_2}]_{j_1,j_2}=\mathbb{P}(\bY_1=(\bx+(i_1,i_2),j_2)\,|\,\bY_0=(\bx,j_1))\ \mbox{for any $\bx\in\mathbb{B}^\alpha$}, 
\]
where we assume the blocks corresponding to impossible transitions are zero (see Fig.\ \ref{fig:fig11}). %
%
\begin{figure}[t]
\begin{center}
\includegraphics[width=55mm,trim=0 0 0 0]{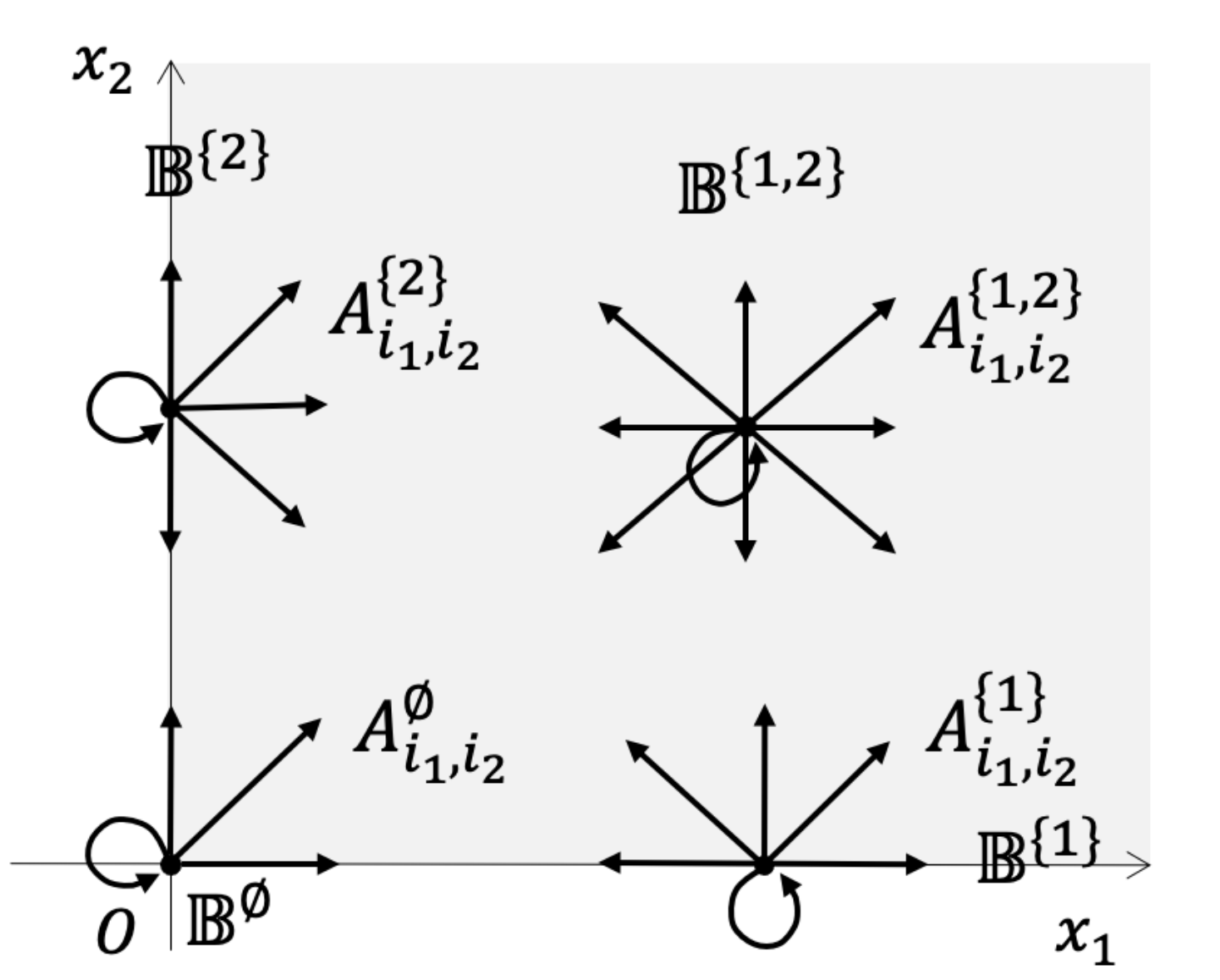} 
\caption{Transition probability blocks}
\label{fig:fig11}
\end{center}
\end{figure}
%
For example, if $\alpha=\{1\}$, we have $A^{\alpha}_{i,-1}=O$ for $i\in\{-1,0,1\}$. 
Since the level process is skip free, for every $\bx,\bx'\in\mathbb{Z}_+^2$, $P_{\bx,\bx'}$ is given by 
\begin{equation}
P_{\bx,\bx'} 
= \left\{ \begin{array}{ll} 
A^\alpha_{\bx'-\bx}, & \mbox{if $\bx\in\mathbb{B}^\alpha$ for some $\alpha\in\scrI_2$ and $\bx'-\bx\in\{-1,0,1\}^2$}, \cr
O, & \mbox{otherwise}.
\end{array} \right.
\end{equation}
We assume the following condition throughout the paper. 
\begin{assumption} \label{as:QBD_irreducible}
The 2d-QBD process $\{\bY_n\}$ is irreducible and aperiodic. 
\end{assumption}

%
Next, we define several Markov chains derived from the 2d-QBD process.
For a nonempty set $\alpha\in\scrI_2$, let $\{\bY^\alpha_n\}=\{(\bX^\alpha_n,J^\alpha_n)\}$ be a process derived from the 2d-QBD process $\{\bY_n\}$ by removing the boundaries that are orthogonal to the $x_i$-axis for each $i\in\alpha$. To be precise, the process $\{\bY^{\{1\}}_n\}$ is a Markov chain on $\mathbb{Z}\times\mathbb{Z}_+\times S_0$ whose transition probability matrix $P^{\{1\}}=(P^{\{1\}}_{\bx,\bx'};\bx,\bx'\in\mathbb{Z}\times\mathbb{Z}_+)$ is given as
\begin{equation}
P^{\{1\}}_{\bx,\bx'} 
= \left\{ \begin{array}{ll} 
A^{\{1\}}_{\bx'-\bx}, & \mbox{if $\bx\in\mathbb{Z}\times\{0\}$ and $\bx'-\bx\in\{-1,0,1\}\times\{0,1\}$}, \cr
A^{\{1,2\}}_{\bx'-\bx}, & \mbox{if $\bx\in\mathbb{Z}\times\mathbb{N}$ and $\bx'-\bx\in\{-1,0,1\}^2$}, \cr
O, & \mbox{otherwise},
\end{array} \right.
\end{equation}
where $\mathbb{N}$ is the set of all positive integers.
The process $\{\bY^{\{2\}}_n\}$ on $\mathbb{Z}_+\times\mathbb{Z}\times S_0$ and its transition probability matrix $P^{\{2\}}=(P^{\{2\}}_{\bx,\bx'};\bx,\bx'\in\mathbb{Z}_+\times\mathbb{Z})$ are analogously defined. The process $\{\bY^{\{1,2\}}_n\}$ is a Markov chain on $\mathbb{Z}^2\times S_0$, whose transition probability matrix $P^{\{1,2\}}=(P^{\{1,2\}}_{\bx,\bx'};\bx,\bx'\in\mathbb{Z}^2)$ is given as
\begin{equation}
P^{\{1,2\}}_{\bx,\bx'} 
= \left\{ \begin{array}{ll} 
A^{\{1,2\}}_{\bx'-\bx}, & \mbox{if $\bx'-\bx\in\{-1,0,1\}^2$}, \cr
O, & \mbox{otherwise}. 
\end{array} \right.
\end{equation}
Regarding $X^{\{1\}}_{1,n}$ as the additive part, we see that the process $\{\bY^{\{1\}}_n\}=\{(X^{\{1\}}_{1,n},(X^{\{1\}}_{2,n},J^{\{1\}}_n))\}$ is a Markov additive process (MA-process for short) with the background state $(X^{\{1\}}_{2,n},J^{\{1\}}_n)$ (see, for example, Ney and Nummelin \cite{Ney87}). The process $\{\bY^{\{2\}}_n\}=\{(X^{\{2\}}_{2,n},(X^{\{2\}}_{1,n},J^{\{2\}}_n))\}$ is also an MA-process, where $X^{\{2\}}_{2,n}$ is the additive part and $(X^{\{2\}}_{1,n},J^{\{2\}}_n)$ the background state, and $\{\bY^{\{1,2\}}_n\}=\{(X^{\{1,2\}}_{1,n},X^{\{1,2\}}_{2,n}),J^{\{1,2\}}_n)\}$ an MA-process, where $(X^{\{1,2\}}_{1,n},X^{\{1,2\}}_{2,n})$ the additive part and $J^{\{1,2\}}_n$ the background state. We call them the induced MA-processes derived from the original 2d-QBD process. Their background processes are called induced Markov chains in Fayolle et.\ al.\ \cite{Fayolle95}. 
%
Let $\{ \bar A^{\{1\}}_i; i\in\{-1,0,1\} \}$ be the Markov additive kernel (MA-kernel for short) of the induced MA-process $\{\bY^{\{1\}}_n\}$, which is the set of transition probability blocks and defined as, for $i\in\{-1,0,1\}$,  
\begin{align*}
& \bar A^{\{1\}}_i = \left( \bar A^{\{1\}}_{i,(x_2,x_2')}; x_2,x_2'\in\mathbb{Z}_+ \right),\\
&\bar A^{\{1\}}_{i,(x_2,x_2')} = \left\{ \begin{array}{ll}
 A^{\{1\}}_{i,x_2'-x_2}, & \mbox{if $x_2=0$ and $x_2'-x_2\in\{0,1\}$}, \cr
 A^{\{1,2\}}_{i,x_2'-x_2}, & \mbox{if $x_2\ge 1$ and $x_2'-x_2\in\{-1,0,1\}$}, \cr
 O, & \mbox{otherwise}. 
 \end{array} \right.
\end{align*}
Let $\{ \bar A^{\{2\}}_i; i\in\{-1,0,1\} \}$ be the MA-kernel of $\{\bY^{\{2\}}_n\}$, defined in the same manner. With respect to $\{\bY^{\{1,2\}}_n\}$, the MA-kernel is given by $\{ A^{\{1,2\}}_{i_1,i_2}; i_1,i_2\in\{-1,0,1\} \}$. 
We assume the following condition throughout the paper. 
\begin{assumption} \label{as:MAprocess_irreducible}
The induced MA-processes $\{\bY^{\{1\}}_n\}$, $\{\bY^{\{2\}}_n\}$ and $\{\bY^{\{1,2\}}_n\}$ are irreducible and aperiodic. 
\end{assumption}

Let $\bar A^{\{1\}}_*(z)$ and  $\bar A^{\{2\}}_*(z)$ be the matrix generating functions of the kernels of $\{\bY^{\{1\}}_n\}$ and $\{\bY^{\{2\}}_n\}$, respectively, defined as
\[
\bar A^{\{1\}}_*(z) = \sum_{i\in\{-1,0,1\}} z^i \bar A^{\{1\}}_i,\quad 
\bar A^{\{2\}}_*(z) = \sum_{i\in\{-1,0,1\}} z^i \bar A^{\{2\}}_i. 
\]
The matrix generating function of the kernel of $\{\bY^{\{1,2\}}_n\}$ is given by $A^{\{1,2\}}_{*,*}(z_1,z_2)$, defined as
\[
A^{\{1,2\}}_{*,*}(z_1,z_2) = \sum_{i_1,i_2\in\{-1,0,1\}} z_1^{i_1} z_2^{i_2} A^{\{1,2\}}_{i_1,i_2}. 
\]
Note that we use generating functions instead of moment generating functions in the paper because the generating functions are more suitable for complex analysis. 
Let $\Gamma^{\{1\}}$, $\Gamma^{\{2\}}$ and $\Gamma^{\{1,2\}}$ be regions in which  the convergence parameters of $\bar A^{\{1\}}_*(e^{\theta_1})$, $\bar A^{\{2\}}_*(e^{\theta_2})$ and $A^{\{1,2\}}_{*,*}(e^{\theta_1},e^{\theta_2})$ are greater than $1$, respectively, i.e., 
\begin{align*}
&\Gamma^{\{1\}} = \{(\theta_1,\theta_2)\in\mathbb{R}^2; \cp(\bar A^{\{1\}}_*(e^{\theta_1}))>1 \},\quad 
\Gamma^{\{2\}} = \{(\theta_1,\theta_2)\in\mathbb{R}^2; \cp(\bar A^{\{2\}}_*(e^{\theta_2}))>1 \}, \\
&\Gamma^{\{1,2\}} = \{(\theta_1,\theta_2)\in\mathbb{R}^2; \cp(A^{\{1,2\}}_{*,*}(e^{\theta_1},e^{\theta_2}))>1 \}. 
\end{align*}
By Lemma A.1 of  Ozawa \cite{Ozawa21}, $\cp(\bar A^{\{1\}}_*(e^\theta))^{-1}$ and $\cp(\bar A^{\{2\}}_*(e^\theta))^{-1}$ are log-convex in $\theta$, and the closures of $\Gamma^{\{1\}}$ and $\Gamma^{\{2\}}$ are convex sets; $\cp(\bar A^{\{1,2\}}_*(e^{\theta_1},e^{\theta_2}))^{-1}$ is also log-convex in $(\theta_1,\theta_2)$, and the closure of $\Gamma^{\{1,2\}}$ is a convex set. Furthermore, by Proposition B.1 of Ozawa \cite{Ozawa21}, $\Gamma^{\{1,2\}}$ is bounded under Assumption \ref{as:MAprocess_irreducible}. 

%
Assuming $\{\bY_n\}$ is positive recurrent (a condition for this will be given in the next section), we denote by $\bnu$ the stationary distribution of $\{\bY_n\}$, where $\bnu=(\bnu_{\bx},\bx\in\mathbb{Z}_+^2)$, $\bnu_{\bx}=(\nu_{(\bx,j)}, j\in S_0)$ and $\nu_{(\bx,j)}$ is the stationary probability that the 2d-QBD process is in the state $(\bx,j)$ in steady state. 
Let $\bc=(c_1,c_2)\in\mathbb{N}^2$ be an arbitrary discrete direction vector and, for $i\in\{1,2\}$, define a real value $\theta_{\bc,i}^\dagger$ as
\begin{equation}
\theta_{\bc,i}^\dagger = \sup\{\langle \bc,\btheta \rangle; \btheta\in\Gamma^{\{i\}}\cap\Gamma^{\{1,2\}}\},  
\label{eq:theta_bcd0}
\end{equation}
where $\langle \ba,\bb \rangle$ is the inner product of vectors $\ba$ ad $\bb$. Our main aim is to demonstrate under certain conditions that, for any $j\in S_0$, 
\begin{equation}
\lim_{k\to\infty} \frac 1 k \log \nu_{(k \bc,j)} = -\min\{\theta_{\bc,1}^\dagger,\,\theta_{\bc,2}^\dagger\},
\label{eq:asympc_eq0}
\end{equation}
i.e., the asymptotic decay rate of the stationary distribution in direction $\bc$ is given by the smaller of $\theta_{\bc,1}^\dagger$ and $\theta_{\bc,2}^\dagger$. We also present a condition ensuring the sequence $\{\nu_{(k \bc,j)}\}_{k\ge 0}$ geometrically decays without power terms. We prove them by using the matrix analytic method of the queueing theory as well as the complex analytic method; the former has been introduced by Marcel Neuts and developed in the literature; See, for example, Refs.\ \cite{Bini05,Latouche99,Neuts94,Neuts89}.
Our model is a kind of multidimensional reflecting process, and asymptotics in various multidimensional reflecting processes have been investigated in the literature for several decades; See Miyazawa \cite{Miyazawa11} and references therein. 
$0$-partially homogeneous ergodic Markov chains discussed in Borovkov and Mogul'ski\u\i\ \cite{Borovkov01} are Markov chains on the positive quadrant including 2d-RRWs as a special case. For those Markov chains, a formula corresponding to \eqref{eq:asympc_eq0} have been obtained by using the large deviations techniques; See Theorem 3.1 of Borovkov and Mogul'ski\u\i\ \cite{Borovkov01} and also see Proposition 5.1 of Miyazawa \cite{Miyazawa09} for the case of 2d-RRW. They have considered only models without background processes. 
In Dai and Miyazawa \cite{Dai13}, results parallel to ours have been obtained for a two-dimensional continuous-state Markov process, named semimartingale-reflecting Brownian motion (SRBM for short). The 2d-SRBM is also a model without background processes. 
With respect to models with a background process, asymptotics of the stationary distribution in a Markov modulated fluid network with a finite number of stations have recently been studied in Miyazawa \cite{Miyazawa21}, where upper and lower bounds for the stationary tail decay rate in various directions were obtained by using so-called Dynkin's formula. 

%
The rest of the paper is organized as follows. In Section \ref{sec:preliminary}, we give a stability condition for the 2d-QBD process and define the asymptotic decay rates of the stationary distribution. In the same section, we introduce a key formula representing the stationary distribution in terms of the fundamental (potential) matrix of the induced MA-process $\{\bY_n^{\{1,2\}}\}$. We call it a compensation equation. Furthermore, we define block state processes derived from the original 2d-QBD process, which will be used for proving propositions in the following sections.  A summary of their properties is given in Appendix \ref{sec:block_2dQBD_results}. 
In Section \ref{sec:asymptotics}, we obtain the asymptotic decay rate of the stationary distribution in any direction. First, we obtain it in the case where the direction vector is given by $\bc=(1,1)$. The asymptotic decay rate for a general direction vector is obtained from the results in the case of $\bc=(1,1)$, by using the block state process. 
In Section \ref{sec:discussion}, we explain a geometric property of the asymptotic decay rates and give an example of two-queue model. In the two-queue model, the asymptotic decay rate corresponds to the decreasing rate of the joint queue length probability in steady state when the queue lengths of both the queues simultaneously enlarge. 
The paper concludes with a remark about the relation between our analysis and the large deviation techniques  in Section \ref{sec:conclusion}.

\medskip
\textit{Notation for vectors and matrices.} 
%
%
%
%
For a matrix $A$, we denote by $[A]_{i,j}$ the $(i,j)$-entry of $A$ and by $A^\top$ the transpose of $A$.  If $A=(a_{i,j})$, $|A|=(|a_{i,j}|)$. Similar notations are also used for vectors. 
The convergence parameter of a nonnegative square matrix $A$ with a finite or countable dimension is denoted by $\cp(A)$, i.e., $\cp(A) = \sup\{r\in\mathbb{R}_+; \sum_{n=0}^\infty r^n A^n<\infty,\ \mbox{entry-wise} \}$. 
For a finite square matrix $A$, we denote by $\spr(A)$ the spectral radius of $A$, which is the maximum modulus of eigenvalue of $A$. If $A$ is nonnegative, $\spr(A)$ corresponds to the Perron-Frobenius eigenvalue of $A$ and we have $\spr(A)=\cp(A)^{-1}$. 
$O$ is a matrix of $0$'s, $\bone$ is a column vector of $1$'s and $\bzero$ is a column vector of $0$'s; their dimensions, which are finite or countably infinite, are determined in context. $I$ is the identity matrix. 
For an $n_1\times n_2$ matrix $A=(a_{i,j})$, $\vec(A)$ is the vector of stacked columns of $A$, i.e., 
$\vec(A)=(a_{1,1},\cdots,a_{n_1,1},a_{1,2},\cdots,a_{n_1,2},\cdots,a_{1,n_2},\cdots,a_{n_1,n_2})^\top$.

%
%
\section{Preliminaries} \label{sec:preliminary}

%
%
\subsection{Stability condition}

Let $a^{\{1\}}$, $a^{\{2\}}$ and $\ba^{\{1,2\}}$ be the mean drifts of the additive part in the induced MA-processes $\{\bY^{\{1\}}_n\}$, $\{\bY^{\{2\}}_n\}$ and $\{\bY^{\{1,2\}}_n\}$, respectively, i.e., 
\[
a^{\{i\}} = \lim_{n\to\infty} \frac{1}{n} \sum_{k=1}^n (X^{\{i\}}_{i,k}-X^{\{i\}}_{i,k-1}),\ i=1,2,\quad
\ba^{\{1,2\}} = \lim_{n\to\infty} \frac{1}{n} \sum_{k=1}^n (\bX^{\{1,2\}}_k-\bX^{\{1,2\}}_{k-1}), 
\]
where $\ba^{\{1,2\}}=(a^{\{1,2\}}_1,a^{\{1,2\}}_2)$. By Corollary 3.1 of Ozawa \cite{Ozawa19}, the stability condition of the 2d-QBD process $\{\bY_n\}$ is given as follows:
\begin{lemma} \label{le:stability_cond}
\begin{itemize}
\item[(i)] In the case where $a^{\{1,2\}}_1<0$ and $a^{\{1,2\}}_2<0$, the 2d-QBD process $\{\bY_n\}$ is positive recurrent if $a^{\{1\}}<0$ and $a^{\{2\}}<0$, and it is transient if either $a^{\{1\}}>0$ or $a^{\{2\}}>0$. 
\item[(ii)] In the case where $a^{\{1,2\}}_1\ge 0$ and $a^{\{1,2\}}_2<0$, $\{\bY_n\}$ is positive recurrent if $a^{\{1\}}<0$, and it is transient if $a^{\{1\}}>0$. 
\item[(iii)] In the case where $a^{\{1,2\}}_1<0$ and $a^{\{1,2\}}_2\ge 0$, $\{\bY_n\}$ is positive recurrent if $a^{\{2\}}<0$, and it is transient if $a^{\{2\}}>0$. 
\item[(iv)] If one of $a^{\{1,2\}}_1$ and $a^{\{1,2\}}_2$ is positive and the other is non-negative, then $\{\bY_n\}$ is transient.
\end{itemize}
\end{lemma}

Each mean drift is represented in terms of the stationary distribution of the corresponding induced Markov chain, i.e., the background process of the corresponding induced MA-process; for their expressions, see Subsection 3.1 of Ozawa \cite{Ozawa19} and its related parts. 
We assume the following condition throughout the paper.
\begin{assumption} \label{as:2dQBD_stable}
The condition in Lemma \ref{le:stability_cond} that ensures the 2d-QBD process $\{\bY_n\}$ is positive recurrent holds.
\end{assumption}

%
\subsection{Compensation equation}

Consider the induced MA-process $\{\bY^{\{1,2\}}_n\}$ on $\mathbb{Z}^2\times S_0$. Its transition probability matrix is given by $P^{\{1,2\}}=(P^{\{1,2\}}_{\bx,\bx'};\bx,\bx'\in\mathbb{Z}^2)$. Denote by $\Phi^{\{1,2\}}=(\Phi^{\{1,2\}}_{\bx,\bx'};\bx,\bx'\in\mathbb{Z}^2)$ the fundamental matrix (potential matrix) of $P^{\{1,2\}}$, i.e., $\Phi^{\{1,2\}}=\sum_{n=0}^\infty (P^{\{1,2\}})^n$.
Under Assumption \ref{as:2dQBD_stable}, since at least one element of the mean drift vector of $\{\bY^{\{1,2\}}_n\}$, $a^{\{1,2\}}_1$ or $a^{\{1,2\}}_2$, is negative, $\Phi^{\{1,2\}}$ is entry-wise finite. 
Since the transition probabilities of $\{\bY^{\{1,2\}}_n\}$ are space-homogeneous with respect to the additive part, we have for every $\bx,\bx'\in\mathbb{Z}^2$ and for every $\bl\in\mathbb{Z}^2$ that 
\begin{equation}
\Phi^{\{1,2\}}_{\bx,\bx'}=\Phi^{\{1,2\}}_{\bx-\bl,\bx'-\bl}. 
\label{eq:Phixx_eq}
\end{equation}
Furthermore, $\Phi^{\{1,2\}}$ satisfies the following property. 
\begin{proposition} \label{pr:Phi12_bounded}
$\Phi^{\{1,2\}}$ is entry-wise bounded. 
%
%
\end{proposition}

\begin{proof}
By \eqref{eq:Phixx_eq}, it suffices to show that, for every $j,j'\in S_0$,
\begin{equation}
\sup_{\bx\in\mathbb{Z}^2} [\Phi^{\{1,2\}}_{\bx,\bzero}]_{j,j'} < \infty,  
\label{eq:Phi12xzer0_bounded}
\end{equation}
where we use the fact that $S_0$ is finite. Let $\tau(j')$ be the first hitting time of $\{\bY^{\{1,2\}}_n\}$ to the state $(0,0,j')$, i.e., 
\[
\tau(j') = \inf\{n\ge 0; \bY^{\{1,2\}}_n=(0,0,j')\}. 
\]
Since $\tau(j')$ is a stopping time, we have by the strong Markov property of $\{\bY^{\{1,2\}}_n\}$ that
\begin{align}
[\Phi^{\{1,2\}}_{\bx,\bzero}]_{j,j'} 
&= \sum_{k=0}^\infty \mathbb{E}\!\left( \sum_{n=0}^\infty 1(\bY^{\{1,2\}}_n=(0,0,j')) \,\Big|\, \tau(j')=k, \bY^{\{1,2\}}_0=(\bx,j) \right) \cr
&\qquad\qquad \cdot \mathbb{P}(\tau(j')=k\,|\,\bY^{\{1,2\}}_0=(\bx,j)) \cr
&= \mathbb{E}\!\left( \sum_{n=0}^\infty 1(\bY^{\{1,2\}}_n=(0,0,j')) \,\Big|\, \bY^{\{1,2\}}_0=(0,0,j') \right) \mathbb{P}(\tau(j')<\infty\,|\,\bY^{\{1,2\}}_0=(\bx,j)) \cr
&\le [\Phi^{\{1,2\}}_{\bzero,\bzero}]_{j',j'}.
\end{align}
Since $\Phi^{\{1,2\}}_{\bzero,\bzero}$ is entry-wise finite, this implies inequality \eqref{eq:Phi12xzer0_bounded}. 
\end{proof}

\begin{remark} \label{rm:Phi12_bounded}
From the proof of the proposition, we see that, for every $(\bx,j),(\bx',j')\in\mathbb{Z}^2\times S_0$, 
\begin{equation}
[\Phi^{\{1,2\}}_{\bx,\bx'}]_{j,j'}\le \max_{j''\in S_0}\, [\Phi^{\{1,2\}}_{\bzero,\bzero}]_{j'',j''}. 
\end{equation}
\end{remark}

%
From $\{\bY^{\{1,2\}}_n\}$, we construct another Markov chain on $\mathbb{Z}^2\times S_0$, denoted by $\{\tilde{\bY}^{\{1,2\}}_n\}$, by replacing the transition probabilities from the states in $\mathbb{B}^\emptyset\cup\mathbb{B}^{\{1\}}\cup\mathbb{B}^{\{2\}}$ with those of the original 2d-QBD process. 
To be precise, the transition probability matrix of $\{\tilde{\bY}^{\{1,2\}}_n\}$, denoted by $\tilde{P}^{\{1,2\}}=(\tilde{P}^{\{1,2\}}_{\bx,\bx'};\bx,\bx'\in\mathbb{Z}^2)$, is given as
\[
\tilde{P}^{\{1,2\}}_{\bx,\bx'}
= \left\{ \begin{array}{ll}
A^\alpha_{\bx'-\bx}, & \mbox{$\bx\in\mathbb{B}^\alpha$ for some $\alpha\in\{\{1\},\{2\},\emptyset\}$ and $\bx'-\bx\in\{-1,0,1\}^2$}, \cr
A^{\{1,2\}}_{\bx'-\bx}, & \mbox{$\bx\notin\mathbb{B}^\alpha$ for any $\alpha\in\{\{1\},\{2\},\emptyset\}$ and $\bx'-\bx\in\{-1,0,1\}^2$}, \cr
O, & \mbox{otherwise}.
\end{array} \right.
\]
The subspace $\mathbb{Z}_+^2\times S_0$ is a unique closed communication class (irreducible class) of the Markov chain $\{\tilde{\bY}^{\{1,2\}}_n\}$ and its stationary distribution, $\tilde{\bnu}=(\tilde{\bnu}_{\bx},\bx\in\mathbb{Z}^2)$, is given as
\[
\tilde{\bnu}_{\bx}=
\left\{ \begin{array}{ll}
\bnu_{\bx}, & \mbox{$\bx\in\mathbb{Z}_+^2$}, \cr
\bzero^\top, &  \mbox{otherwise}, \end{array} \right.
\]
where $\bnu=(\bnu_{\bx},\bx\in\mathbb{Z}_+^2)$ is the stationary distribution of the original 2d-QBD process. The stationary distribution $\tilde{\bnu}$ satisfies the stationary equation $\tilde{\bnu} \tilde{P}^{\{1,2\}}=\tilde{\bnu}$. We have the following.
\begin{proposition} \label{pr:tildenuPhi12_finite}
Under Assumption \ref{as:2dQBD_stable}, $\tilde{\bnu} \Phi^{\{1,2\}}$ is elementwise bounded. 
\end{proposition}

\begin{proof}
Since $\tilde{\bnu}$ is a probability distribution, by Remark \ref{rm:Phi12_bounded}, we have for any $(\bx,j)\in\mathbb{Z}^2\times S_0$ that 
\begin{equation}
[\tilde{\bnu} \Phi^{\{1,2\}}]_{(\bx,j)} \le \max_{j''\in S_0}\, [\Phi^{\{1,2\}}_{\bzero,\bzero}]_{j'',j''}. 
\end{equation}
Hence, the assertion of the proposition holds. 
\end{proof}

Since $\Phi^{\{1,2\}}$ is entry-wise finite, we have 
\begin{equation}
(I-P^{\{1,2\}}) \Phi^{\{1,2\}} = \Phi^{\{1,2\}} (I-P^{\{1,2\}}) = I.
\end{equation}
By Proposition \ref{pr:tildenuPhi12_finite}, we obtain the following.
\begin{lemma}
\begin{equation}
\tilde{\bnu}=\tilde{\bnu}(\tilde{P}^{\{1,2\}}-P^{\{1,2\}})\Phi^{\{1,2\}}. 
\label{eq:compensation_eq}
\end{equation}
\end{lemma}
\begin{proof}
By the Fubini's theorem, we have $\tilde{\bnu}\tilde{P}^{\{1,2\}}\Phi^{\{1,2\}}=\tilde{\bnu}\Phi^{\{1,2\}}<\infty$, elementwise. Hence,
\[
\tilde{\bnu}(\tilde{P}^{\{1,2\}}-P^{\{1,2\}}) \Phi^{\{1,2\}}
= \tilde{\bnu}(I-P^{\{1,2\}}) \Phi^{\{1,2\}}
=\tilde{\bnu}, 
\]
where $\tilde{\bnu}(I-P^{\{1,2\}}) \Phi^{\{1,2\}}$ corresponds to a Riesz decomposition of $\tilde{\bnu}$ in the case where the harmonic function term is equivalent to zero; See Theorem 3.1 of Nummelin \cite{Nummelin84}. 
\end{proof}

Equation \eqref{eq:compensation_eq} can also be derived by the compensation method discussed in Keilson \cite{Keilson79}. We, therefore, call it a compensation equation. Its remarkable point is that the nonzero entries of $\tilde{P}^{\{1,2\}}-P^{\{1,2\}}$ are restricted to the transition probabilities from the states in $\mathbb{B}^\emptyset\cup\mathbb{B}^{\{1\}}\cup\mathbb{B}^{\{2\}}$.
Hence, we immediately obtain, for $\bx\in\mathbb{Z}^2$, 
\begin{align}
\tilde{\bnu}_{\bx} 
&= \sum_{i_1,i_2\in\{-1,0,1\}} \bnu_{(0,0)} (A^\emptyset_{i_1,i_2}-A^{\{1,2\}}_{i_1,i_2}) \Phi^{\{1,2\}}_{(i_1,i_2),\bx} \cr
&\qquad +\sum_{k=1}^\infty\ \sum_{i_1,i_2\in\{-1,0,1\}} \bnu_{(k,0)} (A^{\{1\}}_{i_1,i_2}-A^{\{1,2\}}_{i_1,i_2}) \Phi^{\{1,2\}}_{(k+i_1,i_2),\bx} \cr
&\qquad +\sum_{k=1}^\infty\ \sum_{i_1,i_2\in\{-1,0,1\}} \bnu_{(0,k)} (A^{\{2\}}_{i_1,i_2}-A^{\{1,2\}}_{i_1,i_2}) \Phi^{\{1,2\}}_{(i_1,k+i_2),\bx},
\label{eq:tildenux}
\end{align}
where any $A^\alpha_{i_1,i_2}$ corresponding to impossible transitions is assumed to be zero. Equation \eqref{eq:tildenux} plays a crucial role in the following section.

%
%
\subsection{Asymptotic decay rates} \label{sec:decay_rate}

Let $\bc=(c_1,c_2)\in\mathbb{N}^2$ be an arbitrary discrete direction vector. For $(\bx,j)\in\mathbb{Z}_+^2\times S_0$, define lower  and upper asymptotic decay rates $\underline{\xi}_{\bc}(\bx,j)$ and $\bar\xi_{\bc}(\bx,j)$ as 
\[
\underline{\xi}_{\bc}(\bx,j) = -\limsup_{k\to\infty} \frac{1}{k} \log \nu_{(\bx+k\bc,j)},\quad 
\bar\xi_{\bc}(\bx,j) = -\liminf_{k\to\infty} \frac{1}{k} \log \nu_{(\bx+k\bc,j)}.
\]
By the Cauchy-Hadamard theorem, the radius of convergence of the power series of the sequence $\{\nu_{(\bx+k\bc,j)}\}_{k\ge 0}$ is given by $e^{\underline{\xi}_{\bc}(\bx,j)}$. If $\underline{\xi}_{\bc}(\bx,j)=\bar\xi_{\bc}(\bx,j)$, we denote them by $\xi_{\bc}(\bx,j)$ and call it the asymptotic decay rate. 
Under Assumption \ref{as:MAprocess_irreducible}, the following property holds.
\begin{proposition} \label{pr:xi_equality}
For every $(\bx,j),(\bx',j')\in\mathbb{N}^2\times S_0$, $\underline{\xi}_{\bc}(\bx,j)=\underline{\xi}_{\bc}(\bx',j')$ and $\bar\xi_{\bc}(\bx,j)=\bar\xi_{\bc}(\bx',j')$. 
\end{proposition}

Since the proof of the proposition is elementary,  we give it in Appendix \ref{sec:xi_equality}. Hereafter, we denote $\underline{\xi}_{\bc}(\bx,j)$, $\bar\xi_{\bc}(\bx,j)$ and $\xi_{\bc}(\bx,j)$ by $\underline{\xi}_{\bc}$, $\bar\xi_{\bc}$ and $\xi_{\bc}$, respectively. The asymptotic decay rates in the coordinate directions, denoted by $\xi_{(1,0)}$ and $\xi_{(0,1)}$, have already been obtained in Ozawa \cite{Ozawa13}. 

%
%
Let $\{\check\bY^{\{1,2\}}_{n}\}=\{(\check{\bX}^{\{1,2\}}_n,\check{J}^{\{1,2\}}_n)\}$ be a lossy Markov chain derived from the induced MA-process $\{\bY^{\{1,2\}}_n\}$ by restricting the state space of the additive part to $\mathbb{N}^2$. To be precise, the process $\{\check\bY^{\{1,2\}}_n\}$ is a Markov chain on $\mathbb{N}^2\times S_0$ whose transition probability matrix $\check{P}^{\{1,2\}}=(\check{P}^{\{1,2\}}_{\bx,\bx'};\bx,\bx'\in\mathbb{N}^2)$ is given as
\begin{equation}
\check{P}^{\{1,2\}}_{\bx,\bx'} 
= \left\{ \begin{array}{ll} 
A^{\{1,2\}}_{\bx'-\bx}, & \mbox{if ($\bx\in(\mathbb{N}\setminus\{1\})^2$ and $\bx'-\bx\in\{-1,0,1\}^2$)} \cr
& \quad\mbox{or ($\bx\in\{1\}\times\mathbb{N}$ and $\bx'-\bx\in\{0,1\}\times\{-1,0,1\}$)}, \cr
& \quad\mbox{or ($\bx\in\mathbb{N}\times\{1\}$ and $\bx'-\bx\in\{-1,0,1\}\times\{0,1\}$)}, \cr
O, & \mbox{otherwise},
\end{array} \right. 
\end{equation}
where $\check{P}^{\{1,2\}}$ is strictly substochastic. 
%
Let $\check{\Phi}^{\{1,2\}}=(\check{\Phi}^{\{1,2\}}_{\bx,\bx'};\bx,\bx'\in\mathbb{N}^2)$ be the fundamental matrix (potential matrix) of $\check{P}^{\{1,2\}}$, i.e., 
\[
\check{\Phi}^{\{1,2\}} = \sum_{n=0}^\infty (\check{P}^{\{1,2\}})^n.
\]
We assume the following condition throughout the paper.
\begin{assumption} \label{as:Y12_onZpZp_irreducible}
$\{\check{\bY}^{\{1,2\}}_n\}$ is irreducible and aperiodic. 
\end{assumption}
This condition implies that the induced MA-process $\{\bY^{\{1,2\}}_n\}$ is irreducible and aperiodic, cf.\ Assumption \ref{as:MAprocess_irreducible}. 
By Theorem 5.1 of Ozawa \cite{Ozawa21}, we have, for any direction vector $\bc\in\mathbb{N}^2$, every $\bx=(x_1,x_2)\in\mathbb{N}^2$ such that $x_1=1$ or $x_2=1$, every $\bl\in\mathbb{Z}_+^2$ and  every $j_1,j_2\in S_0$, 
\begin{equation}
\lim_{k\to\infty} \frac{1}{k} \log\, [\check{\Phi}^{\{1,2\}}_{\bx,k\bc+\bl}]_{j_1,j_2} = - \sup\{\langle \bc,\btheta \rangle; \btheta\in\Gamma^{\{1,2\}}\}. 
\label{eq:tildePhi_limit}
\end{equation}
Since the stationary distribution of the 2d-QBD process can be represented in terms of the entries of $\check{\Phi}^{\{1,2\}}$ (see Section 6 of Ozawa \cite{Ozawa21}), this formula leads us to the following.  
\begin{lemma} \label{le:decayrate_upper}
\begin{equation}
\bar\xi_{\bc} \le \sup\{\langle \bc,\btheta \rangle; \btheta\in\Gamma^{\{1,2\}}\}.  
\label{eq:xic_upper}
\end{equation}
\end{lemma}

%
%
%
\subsection{Block state process} \label{sec:block_state_process}

For $\bb=(b_1,b_2)\in\mathbb{N}^2$, we consider another 2d-QBD process derived from the original 2d-QBD process $\{\bY_n\}=\{(\bX_n,J_n)\}$ by regarding each $b_1\times b_2$ block of level as a level (see, for example, Subsection 4.2 of Ozawa \cite{Ozawa21}).
For $i\in\{1,2\}$, denote by ${}^{\bb}\!X_{i,n}$ and ${}^{\bb}\!M_{i,n}$ the quotient and remainder of $X_{i,n}$ divided by $b_i$, respectively, i.e., 
\[
X_{i,n}=b_i {}^{\bb}\!X_{i,n}+{}^{\bb}\!M_{i,n},
\]
where ${}^{\bb}\!X_{i,n}\in\mathbb{Z}_+$ and ${}^{\bb}\!M_{i,n}\in\{0,1,...,b_i-1\}$. 
Define a process $\{{}^{\bb}\bY_n\}$ as 
\[
{}^{\bb}\bY_n=({}^{\bb}\!\bX_n,({}^{\bb}\!\bM_n,{}^{\bb}\!J_n)),
\]
where ${}^{\bb}\!\bX_n=({}^{\bb}\!X_{1,n},{}^{\bb}\!X_{2,n})$ is the level state and $({}^{\bb}\!\bM_n,{}^{\bb}\!J_n)=({}^{\bb}\!M_{1,n},{}^{\bb}\!M_{2,n},J_n)$ the phase state. The process $\{{}^{\bb}\bY_n\}$ is a 2d-QBD process and its state space is given by $\mathbb{Z}_+^2\times(\mathbb{Z}_{0,b_1-1}\times\mathbb{Z}_{0,b_2-1}\times S_0)$, where $\mathbb{Z}_{0,b_i-1}=\{0,1,...,b_i-1\}$. We call $\{{}^{\bb}\bY_n\}$ a $\bb$-block state process. 
The transition probability matrix of $\{{}^{\bb}\bY_n\}$, denoted by ${}^{\bb}\!P=({}^{\bb}\!P_{\bx,\bx'};\bx,\bx'\in\mathbb{Z}_+^2)$, has the same block structure as $P$. For $\alpha\in\scrI_2$ and $i_1,i_2\in\{-1,0,1\}$, denote by  ${}^{\bb}\!A_{i_1,i_2}^\alpha$ the transition probability block of ${}^{\bb}\!P$ corresponding to $A_{i_1,i_2}^\alpha$ of $P$, then ${}^{\bb}\!A_{i_1,i_2}^\alpha$ can be represented by using $A_{i'_1,i'_2}^{\alpha'},\,\alpha'\in\scrI_2,\,i'_1,i'_2\in\{-1,0,1\}$. We omit the explicit expressions for the transition probability blocks since we do not use them directly. 
Let ${}^{\bb}\bnu=({}^{\bb}\bnu_{\bx}; \bx=(x_1,x_2)\in\mathbb{Z}_+^2)$ be the stationary distribution of $\{{}^{\bb}\bY_n\}$, where 
\[
{}^{\bb}\bnu_{\bx}=\left(\bnu_{(b_1x_1+i_1,b_2x_2+i_2)};i_1\in\mathbb{Z}_{0,b_1-1},\,i_2\in\mathbb{Z}_{0,b_2-1}\right) 
\]
and $\bnu=(\bnu_{\bx}; \bx\in\mathbb{Z}_+^2)$ is the stationary distribution of the original 2d-QBD process. 
Denote by ${}^{\bb}\xi_{(1,0)}$ and ${}^{\bb}\xi_{(0,1)}$ the asymptotic decay rates of the sequences $\{{}^{\bb}\bnu_{(k,0)}\}_{k\ge 0}$ and $\{{}^{\bb}\bnu_{(0,k)}\}_{k\ge 0}$, respectively, i.e., for $i_1\in\mathbb{Z}_{0,b_1-1}$, $i_2\in\mathbb{Z}_{0,b_2-1}$ and $j\in S_0$, 
\[
{}^{\bb}\xi_{(1,0)} = -\lim_{k\to\infty} \frac{1}{k} \log \bigl[ [{}^{\bb}\bnu_{(k,0)}]_{i_1,i_2} \bigr]_j,\quad 
{}^{\bb}\xi_{(0,1)} = -\lim_{k\to\infty} \frac{1}{k} \log \bigl[ [{}^{\bb}\bnu_{(0,k)}]_{i_1,i_2} \bigr]_j, 
\]
where ${}^{\bb}\xi_{(1,0)}$ and ${}^{\bb}\xi_{(0,1)}$ do not depend on any of $i_1$, $i_2$ and $j$. 
Since $\{{}^{\bb}\bY_n\}$ is a 2d-QBD process and inherits the nature of the original 2d-QBD process, the results obtained in Refs.\ \cite{Ozawa13,Ozawa18} also hold for $\{{}^{\bb}\bY_n\}$. For example, we have ${}^{\bb}\xi_{(1,0)}=b_1 \xi_{(1,0)}$ and ${}^{\bb}\xi_{(0,1)}=b_2 \xi_{(0,1)}$. 
For later use, we summarize the properties of $\{{}^{\bb}\bY_n\}$ in Appendix \ref{sec:block_2dQBD_results}, and here define, for $\bc\in\mathbb{N}^2$, the asymptotic decay rate in direction $\bc$, ${}^{\bb}\xi_{\bc}$, as 
\[
{}^{\bb}\xi_{\bc} = -\lim_{k\to\infty} \frac{1}{k} \log \bigl[ [{}^{\bb}\bnu_{k\bc}]_{i_1,i_2} \bigr]_j, 
\]
where $i_1$, $i_2$ and $j$ are arbitrary.

%
%

\section{Asymptotics in an arbitrary direction} \label{sec:asymptotics}

Hereafter, we use the following notation: For $r>0$, $\Delta_r$ and $\partial\Delta_r$ are the open disk and circle of center $0$ and radius $r$ on the complex plane, respectively. For $r_1,r_2>0$ such that $r_1<r_2$, $\Delta_{r_1,r_2}$ is the open annular domain defined as $\Delta_{r_1,r_2}=\Delta_{r_2}\setminus(\Delta_{r_1}\cup\partial\Delta_{r_1})$. 
%

\subsection{Methodology and preparation}

For $\bc=(c_1,c_2)\in\mathbb{N}^2$, define the generating function of the stationary probabilities of the 2d-QBD process in direction $\bc$, $\bvarphi^{\bc}(z)$, as
\[
\bvarphi^{\bc}(z) 
= \sum_{k=0}^\infty z^k \bnu_{k\bc} 
=  \sum_{k=-\infty}^\infty z^k \tilde{\bnu}_{k\bc}, 
\]
where $z$ is a complex variable. 
Furthermore, define real values $\theta_{\bc}^{min}$ and $\theta_{\bc}^{max}$ as
\[
\theta_{\bc}^{min}=\inf\{\langle \bc, \btheta \rangle; \btheta\in\Gamma^{\{1,2\}} \},\quad
\theta_{\bc}^{max}=\sup\{\langle \bc, \btheta \rangle; \btheta\in\Gamma^{\{1,2\}} \}. 
\]
For any $x\in[0,e^{\theta_{\bc}^{max}})$, if the power series $\bvarphi^{\bc}(z)$ is absolutely convergent at $z=x$, then we have by the Caucy-Hadamard theorem that $\underline{\xi}_{\bc}\ge \theta_{\bc}^{max}$. This and \eqref{eq:xic_upper} imply $\xi_{\bc}=\theta_{\bc}^{max}$. We use this procedure for obtaining the asymptotic decay rate $\xi_{\bc}$ when $\xi_{\bc}$ is given by $\theta_{\bc}^{max}$. 

On the other hand, when $\xi_{\bc}$ is less than $\theta_{\bc}^{max}$, we demonstrate that for a certain point $x_0\in[0,e^{\theta_{\bc}^{max}})$,
\begin{equation}
\lim_{x\,\uparrow\, x_0} (x_0-x) \bvarphi(x)= \bg \quad \mbox{for some positive vector $\bg$}, 
\end{equation}
and that the complex function $\bvarphi^{\bc}(z)$ is analytic in $\Delta_{x_0}\cup(\partial\Delta_{x_0}\setminus\{x_0\})$. In this case, by Theorem VI.4 of Flajolet and Sedgewick \cite{Flajolet09}, the exact asymptotic formula for the sequence $\{\bnu_{k\bc}\}$ is given by $x_0^{-k}$ and we have $\xi_{\bc}=\log x_0$. 

For $k\in\mathbb{Z}$, let $\mathbb{Z}_{\le k}$ and $\mathbb{Z}_{\ge k}$ be the set of integers less than or equal to $k$ and that of integers greater than or equal to $k$, respectively. We introduce additional assumptions.
\begin{assumption} \label{as:Y12_onZpmZmp_irreducible}
\begin{itemize}
\item[(i)] The lossy Markov chain derived from the induced MA-process $\{\bY^{\{1,2\}}_n\}$ by restricting the state space to $\mathbb{Z}_{\le 0}\times \mathbb{Z}_{\ge 0}\times S_0$ is irreducible and aperiodic. 
\item[(ii)] The lossy Markov chain derived from $\{\bY^{\{1,2\}}_n\}$ by restricting the state space to $\mathbb{Z}_{\ge 0}\times \mathbb{Z}_{\le 0}\times S_0$ is irreducible and aperiodic. 
\end{itemize}
\end{assumption}

\begin{remark} \label{re:Y12_onZpmZmp_irreducible}
Let $\bc=(c_1,c_2)\in\mathbb{N}^2$ be a direction vector, and define subspaces $\mathbb{S}_{\bc}^L$ and $\mathbb{S}_{\bc}^R$ as 
\begin{align*}
&\mathbb{S}_{\bc}^L = \{(x_1,x_2,j)\in\mathbb{Z}^2\times S_0: c_1 x_2-c_2 x_1 \ge 0\}, \\
&\mathbb{S}_{\bc}^R = \{(x_1,x_2,j)\in\mathbb{Z}^2\times S_0: c_1 x_2-c_2 x_1 \le 0\}.
\end{align*}
$\mathbb{S}_{\bc}^L$ is the upper-left space of the line $c_1 x_2-c_2 x_1=0$, and $\mathbb{S}_{\bc}^R$ the lower-right space of the same line. Due to the space-homogeneity of $\{\bY^{\{1,2\}}_n\}$ with respect to the additive part, under part (i) of Assumption \ref{as:Y12_onZpmZmp_irreducible}, the lossy Markov chain derived from $\{\bY^{\{1,2\}}_n\}$ by restricting the state space to $\mathbb{S}_{\bc}^L$ is irreducible and aperiodic, and under part (ii) of Assumption \ref{as:Y12_onZpmZmp_irreducible}, that derived from $\{\bY^{\{1,2\}}_n\}$ by restricting the state space to $\mathbb{S}_{\bc}^R$ is irreducible and aperiodic. 
We will use this property later.
\end{remark}

Assumption \ref{as:Y12_onZpmZmp_irreducible} seems rather strong and it can probably be replaced with other weaker one. We adopt the assumption since it makes discussions simple; See Remark \ref{re:hatGc_eigen} in the following subsection.

%
For $\bx\in\mathbb{Z}^2$, define the matrix generating function of the blocks of $\Phi^{\{1,2\}}$ in direction $\bc$, $\Phi^{\bc}_{\bx,*}(z)$, as
\[
\Phi^{\bc}_{\bx,*}(z)
=  \sum_{k=-\infty}^\infty z^k \Phi^{\{1,2\}}_{\bx,k\bc}.
\]
The matrix generating function $\Phi^{\bc}_{\bx,*}(z)$ satisfies the following.
\begin{proposition} \label{pr:Phic_finite}
For every $\bx\in\mathbb{Z}^2$, $\Phi^{\bc}_{\bx,*}(z)$ is absolutely convergent and entry-wise analytic in the open annual domain $\Delta_{e^{\theta_{\bc}^{min}},e^{\theta_{\bc}^{max}}}$.  
\end{proposition}
\begin{proof}
For every $\bx\in\mathbb{Z}^2$, we have for any $\btheta=(\theta_1,\theta_2)\in\Gamma^{\{1,2\}}$ that
\begin{align}
\infty>\sum_{n=0}^\infty (A^{\{1,2\}}_{*,*}(e^{\theta_1},e^{\theta_2}))^n 
&= \sum_{\bx'\in\mathbb{Z}^2} e^{\langle \bx',\btheta \rangle} \Phi^{\{1,2\}}_{\bzero,\bx'} \cr
&\ge \sum_{k=-\infty}^\infty e^{\langle k\bc-\bx,\btheta \rangle} \Phi^{\{1,2\}}_{\bzero,k\bc-\bx}
= e^{-\langle \bx,\btheta \rangle} \Phi^{\bc}_{\bx,*}(e^{\langle \bc,\btheta \rangle}),  
\end{align}
where we use the identity $\Phi^{\{1,2\}}_{\bzero,k\bc-\bx}=\Phi^{\{1,2\}}_{\bx,k\bc}$. Since the closure of $\Gamma^{\{1,2\}}$  is a convex set, $\Phi^{\bc}_{\bx,*}(z)$ is, therefore, absolutely convergent in $\Delta_{e^{\theta_{\bc}^{min}},e^{\theta_{\bc}^{max}}}$. As a result, $\Phi^{\bc}_{\bx,*}(z)$ is analytic in $\Delta_{e^{\theta_{\bc}^{min}},e^{\theta_{\bc}^{max}}}$ since each entry of $\Phi^{\bc}_{\bx,*}(z)$ is represented as a Laurent series of $z$ (see, for example, Section II.1 of Markushevich \cite{Markushevich05}). 
\end{proof}

By compensation equation \eqref{eq:tildenux}, $\bvarphi^{\bc}(z)$ is given in terms of $\Phi^{\bc}_{\bx,*}(z)$ by 
\begin{equation}
\bvarphi^{\bc}(z) = \bvarphi^{\bc}_0(z) + \bvarphi^{\bc}_1(z) + \bvarphi^{\bc}_2(z), 
\label{eq:varphic_eq1}
\end{equation}
where
\begin{align}
&\bvarphi^{\bc}_0(z) = \sum_{i_1,i_2\in\{-1,0,1\}} \bnu_{(0,0)} (A^\emptyset_{i_1,i_2}-A^{\{1,2\}}_{i_1,i_2}) \Phi^{\bc}_{(i_1,i_2),*}(z), \\
&\bvarphi^{\bc}_1(z) = \sum_{k=1}^\infty\ \sum_{i_1,i_2\in\{-1,0,1\}} \bnu_{(k,0)} (A^{\{1\}}_{i_1,i_2}-A^{\{1,2\}}_{i_1,i_2}) \Phi^{\bc}_{(k+i_1,i_2),*}(z), \\
&\bvarphi^{\bc}_2(z) = \sum_{k=1}^\infty\ \sum_{i_1,i_2\in\{-1,0,1\}} \bnu_{(0,k)} (A^{\{2\}}_{i_1,i_2}-A^{\{1,2\}}_{i_1,i_2}) \Phi^{\bc}_{(i_1,k+i_2),*}(z). 
\label{eq:phic2_def}
\end{align}
For $i\in\{1,2,3\}$, let $\xi_{\bc,i}$ be the supremum point of the convergence domain of $\bvarphi^{\bc}_i(e^\theta)$, defined as
\[
\xi_{\bc,i} = \sup\{\theta\in\mathbb{R}; \mbox{$\bvarphi^{\bc}_i(e^\theta)$ is absolutely convergent, elementwise} \}.
\]
By Proposition \ref{pr:Phic_finite}, we immediately obtain the following.
\begin{proposition} \label{pr:varphic0_finite}
$\bvarphi_0^{\bc}(z)$ is absolutely convergent and elementwise analytic in $\Delta_{e^{\theta_{\bc}^{min}},e^{\theta_{\bc}^{max}}}$, and we have $\xi_{\bc,0}\ge \theta_{\bc}^{max}$.
\end{proposition}

We analyze $\bvarphi_1^{\bc}(z)$ and $\bvarphi_2^{\bc}(z)$ in the following subsection.

%
%
\subsection{In the case of direction vector $\bc=(1,1)$}

In overall this subsection, we assume $\bc=(1,1)$. 

%
First, focusing on $\bvarphi_2^{\bc}(z)$, we construct a new skip-free MA-process
from $\{\bY^{\{1,2\}}_n\}$ and apply the matrix analytic method to the MA-process. 
The new MA-process is denoted by $\{\hat{\bY}_n\}=\{(\hat{\bX}_n,\hat{\bJ}_n)\}=\{(\hat{X}_{1,n},\hat{X}_{2,n}),(\hat{R}_n,\hat{J}_n)\}$, where $\hat{X}_{1,n}=X^{\{1,2\}}_{1,n}$, $\hat{X}_{2,n}$ and $\hat{R}_n$ are the quotient and remainder of $X^{\{1,2\}}_{2,n}-X^{\{1,2\}}_{1,n}$ divided by $2$, respectively, and $\hat{J}_n=J^{\{1,2\}}_n$ (see Fig.\,\ref{fig:fig31}). 
%
\begin{figure}[t]
\begin{center}
\includegraphics[width=80mm,trim=0 0 0 0]{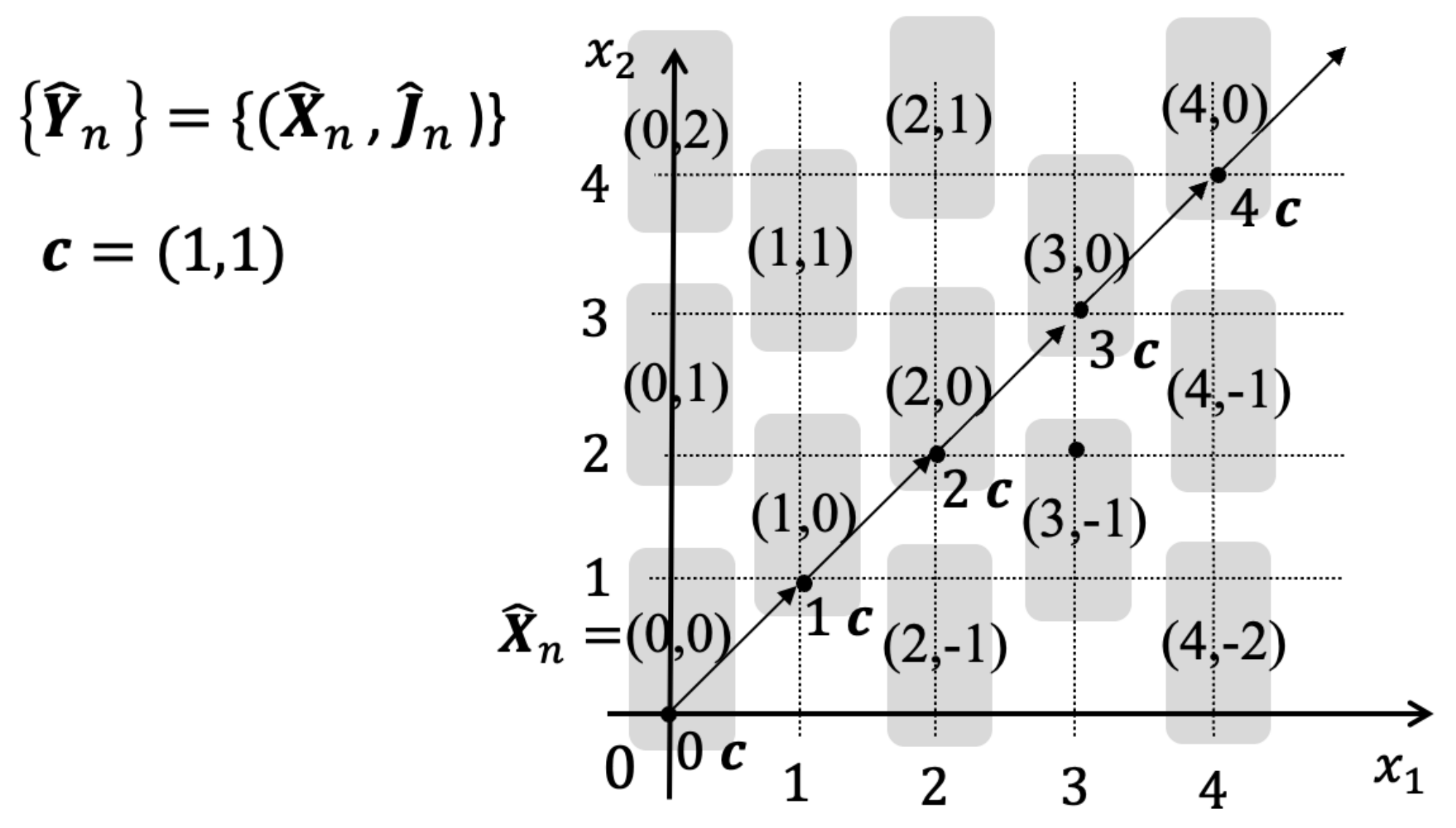} 
\caption{Level space of $\{\hat{\bY}_n\}$}
\label{fig:fig31}
\end{center}
\end{figure}
%
The state space of $\{\hat{\bY}_n\}$ is $\mathbb{Z}^2\times\{0,1\}\times S_0$, and $\hat{\bX}_n$ and $\hat{\bJ}_n$ are the additive part (level state) and background state (phase state), respectively. The additive part of $\{\hat{\bY}_n\}$ is skip free, and this is a reason why we consider this new MA-process. 
From the definition, if $\hat{\bX}_n=(x_1,x_2)$ and $\hat{R}_n=r$ in the new MA-process, it follows that $X^{\{1,2\}}_{1,n}=x_1$,  $X^{\{1,2\}}_{2,n}=x_1+2 x_2+r$ in the original MA-process. Hence, $\hat{\bY}_n=(k,0,0,j)$ means $\bY^{\{1,2\}}_n=(k,k,j)$.
Here we note that $\{\hat{\bY}_n\}$ is slightly different from the $(1,2)$-block state process derived from $\{\bY^{\{1,2\}}_n\}$; See Subsection \ref{sec:block_state_process}. In the latter process, the state $(x_1,x_2,r,j)$ corresponds to the state $(x_1, 2x_2+r,j)$ of the original MA-process. 
Denote by $\hat{P}=(\hat{P}_{\bx,\bx'};\bx,\bx'\in\mathbb{Z}^2)$ the transition probability matrix of $\{\hat{\bY}_n\}$, which is given as 
\[
\hat{P}_{\bx,\bx'} 
= \left\{ \begin{array}{ll} 
\hat{A}^{\{1,2\}}_{\bx'-\bx}, & \mbox{if $\bx'-\bx\in\{-1,0,1\}^2$}, \cr
O, & \mbox{otherwise}, 
\end{array} \right.
\]
where 
\begin{align*}
&\hat{A}^{\{1,2\}}_{-1,1} = \begin{pmatrix} A^{\{1,2\}}_{-1,1} & O \cr  A^{\{1,2\}}_{-1,0} & A^{\{1,2\}}_{-1,1} \end{pmatrix}, \quad
\hat{A}^{\{1,2\}}_{0,1} = \begin{pmatrix} O & O \cr A^{\{1,2\}}_{0,1} & O \end{pmatrix}, \quad
\hat{A}^{\{1,2\}}_{1,1} = \begin{pmatrix} O & O \cr O &O \end{pmatrix}, \cr
&\hat{A}^{\{1,2\}}_{-1,0} = \begin{pmatrix} A^{\{1,2\}}_{-1,-1} & A^{\{1,2\}}_{-1,0} \cr  O & A^{\{1,2\}}_{-1,-1} \end{pmatrix}, \quad
\hat{A}^{\{1,2\}}_{0,0} = \begin{pmatrix} A^{\{1,2\}}_{0,0} & A^{\{1,2\}}_{0,1} \cr A^{\{1,2\}}_{0,-1} & A^{\{1,2\}}_{0,0} \end{pmatrix}, \quad
\hat{A}^{\{1,2\}}_{1,0} = \begin{pmatrix} A^{\{1,2\}}_{1,1} & O \cr A^{\{1,2\}}_{1,0} & A^{\{1,2\}}_{1,1} \end{pmatrix}, \cr
&\hat{A}^{\{1,2\}}_{-1,-1} = \begin{pmatrix} O & O \cr  O & O \end{pmatrix}, \quad
\hat{A}^{\{1,2\}}_{0,-1} = \begin{pmatrix} O & A^{\{1,2\}}_{0,-1} \cr O & O \end{pmatrix}, \quad
\hat{A}^{\{1,2\}}_{1,-1} = \begin{pmatrix} A^{\{1,2\}}_{1,-1} & A^{\{1,2\}}_{1,0} \cr O & A^{\{1,2\}}_{1,-1} \end{pmatrix}.
\end{align*}
Denote by $\hat{\Phi}=(\hat{\Phi}_{\bx,\bx'};\bx,\bx'\in\mathbb{Z}^2)$ the fundamental matrix of $\hat{P}$, i.e., $\hat{\Phi}=\sum_{n=0}^\infty (\hat{P})^n$, and for $\bx=(x_1,x_2)\in\mathbb{Z}^2$, define a matrix generating function $\hat{\Phi}_{\bx,*}(z)$ as
\begin{equation}
\hat{\Phi}_{\bx,*}(z) 
= \sum_{k=-\infty}^\infty z^k \hat{\Phi}_{\bx,(k,0)}
= \begin{pmatrix} 
\Phi^{\bc}_{(x_1,x_1+2 x_2),*}(z) &  \Phi^{\bc}_{(x_1,x_1+2 x_2-1),*}(z) \cr 
\Phi^{\bc}_{(x_1,x_1+2 x_2+1),*}(z) &  \Phi^{\bc}_{(x_1,x_1+2 x_2),*}(z)
\end{pmatrix}.
\label{eq:hatPhixs_def}
\end{equation}
By Proposition \ref{pr:Phic_finite}, for every $\bx\in\mathbb{Z}^2$, $\hat\Phi_{\bx,*}(z)$ is entry-wise analytic in the open annual domain $\Delta_{e^{\theta_{\bc}^{min}},e^{\theta_{\bc}^{max}}}$.  
%
%
Define blocks $\hat{A}^{\{2\}}_{i_1,i_2},\,i_1,i_2\in\{-1,0,1\},$ as $\hat{A}^{\{2\}}_{-1,1} = \hat{A}^{\{2\}}_{-1,0} = \hat{A}^{\{2\}}_{-1,-1} = O$ and 
\begin{align*}
\hat{A}^{\{2\}}_{0,1} = \begin{pmatrix} O & O \cr A^{\{2\}}_{0,1} & O \end{pmatrix}, \quad
\hat{A}^{\{2\}}_{0,0} = \begin{pmatrix} A^{\{2\}}_{0,0} & A^{\{2\}}_{0,1} \cr A^{\{2\}}_{0,-1} & A^{\{2\}}_{0,0} \end{pmatrix}, \quad
\hat{A}^{\{2\}}_{0,-1} = \begin{pmatrix} O & A^{\{2\}}_{0,-1} \cr O & O \end{pmatrix}, \cr
\hat{A}^{\{2\}}_{1,1} = \begin{pmatrix} O & O \cr O &O \end{pmatrix}, \quad 
\hat{A}^{\{2\}}_{1,0} = \begin{pmatrix} A^{\{2\}}_{1,1} & O \cr A^{\{2\}}_{1,0} & A^{\{2\}}_{1,1} \end{pmatrix}, \quad 
\hat{A}^{\{2\}}_{1,-1} = \begin{pmatrix} A^{\{2\}}_{1,-1} & A^{\{2\}}_{1,0} \cr O & A^{\{2\}}_{1,-1} \end{pmatrix}.
\end{align*}
For $i_1,i_2\in\{-1,0,1\}$, define the following matrix generating functions:
\begin{align*}
&\hat{A}^{\{1,2\}}_{*,i_2}(z)=\sum_{i\in\{-1,0,1\}} z^i \hat{A}^{\{1,2\}}_{i,i_2},\quad
\hat{A}^{\{1,2\}}_{i_1,*}(z)=\sum_{i\in\{-1,0,1\}} z^i \hat{A}^{\{1,2\}}_{i_1,i},\\
&\hat{A}^{\{2\}}_{*,i_2}(z)=\sum_{i\in\{0,1\}} z^i \hat{A}^{\{2\}}_{i,i_2}, \quad
\hat{A}^{\{2\}}_{i_1,*}(z)=\sum_{i\in\{-1,0,1\}} z^i \hat{A}^{\{2\}}_{i_1,i}.
\end{align*}
Define a vector generating function $\hat{\bvarphi}^{\bc}_2(z) $ as
\begin{align}
\hat{\bvarphi}^{\bc}_2(z) 
&=\begin{pmatrix} \hat{\bvarphi}^{\bc}_{2,1}(z) & \hat{\bvarphi}^{\bc}_{2,2}(z) \end{pmatrix} \cr
&= \sum_{l=-\infty}^\infty z^l \sum_{k=1}^\infty\ \sum_{i_1,i_2\in\{-1,0,1\}} \hat{\bnu}_{(0,k)} (\hat{A}^{\{2\}}_{i_1,i_2}-\hat{A}^{\{1,2\}}_{i_1,i_2}) \hat{\Phi}_{(i_1,k+i_2),(l,0)}, 
\label{eq:hatphic2_def}
\end{align}
where, for $\bx=(x_1,x_2)\in\mathbb{Z}^2$, 
\[
\hat{\bnu}_{\bx}=\begin{pmatrix} \tilde{\bnu}_{(x_1,x_1+2x_2)} & \tilde{\bnu}_{(x_1,x_1+2x_2+1)} \end{pmatrix}
\]
and hence, for $k\ge 0$, $\hat{\bnu}_{(0,k)}=\begin{pmatrix} \bnu_{(0,2 k)} & \bnu_{(0,2 k+1)} \end{pmatrix}$. Since
\begin{align}
\hat{\bvarphi}^{\bc}_{2,1}(z) 
&= \bvarphi^{\bc}_2(z) - \sum_{i_1,i_2\in\{-1,0,1\}} \bnu_{(0,1)} (A^{\{2\}}_{i_1,i_2}-A^{\{1,2\}}_{i_1,i_2}) \Phi^{\bc}_{(i_1,1+i_2),*}(z), 
\label{eq:hatphic2_eq21} \\
\hat{\bvarphi}^{\bc}_{2,2}(z) 
&= \sum_{k=2}^\infty\ \sum_{i_1,i_2\in\{-1,0,1\}} \bnu_{(0,k)} (A^{\{2\}}_{i_1,i_2}-A^{\{1,2\}}_{i_1,i_2}) \Phi^{\bc}_{(i_1,k+i_2-1),*}(z),  
\label{eq:hatphic2_eq22}
\end{align}
we analyze $\hat{\bvarphi}^{\bc}_2(z)$ instead of $\bvarphi^{\bc}_2(z)$.
Let $\hat\xi_{\bc,2}$ be the supremum point of the convergence domain of $\hat{\bvarphi}^{\bc}_2(e^\theta)$, defined as 
\[
\hat\xi_{\bc,2} = \sup\{\theta\in\mathbb{R}; \mbox{$\hat{\bvarphi}^{\bc}_2(e^\theta)$ is absolutely convergent, elementwise} \}.
\]
By \eqref{eq:hatphic2_eq21}, we have $\xi_{\bc,2}\ge \hat\xi_{\bc,2}$. 

%
%
Next, we obtain a tractable representation for $\hat{\varphi}^{\bc}_2(z)$.  Since  $\{\hat{\bY}_n\}$ is space-homogeneous with respect to the additive part, we have, for every $\bx=(x_1,x_2)\in\mathbb{Z}^2$, 
\begin{equation}
\hat{\Phi}_{(x_1,x_2),*}(z) = z^{x_1} \hat{\Phi}_{(0,x_2),*}(z). 
\label{eq:hatPhi_eq1}
\end{equation}
Define a stopping time $\tau_0$ as 
\[
\tau_0 = \inf\{n\ge 1; \hat{X}_{2,n}=0\}.
\]
This $\tau_0$ is the first hitting time to the subspace $\hat{\mathbb{S}}_0=\{(x_1,x_2,r,j)\in\mathbb{Z}^2\times\{0,1\}\times S_0; x_2=0\}$, which corresponds to the subspace $\{(x_1,x_2,j)\in\mathbb{Z}^2\times S_0; x_2=x_1\ \mbox{or}\ x_2=x_1+1\}$ of the original MA-process (see Fig.\,\ref{fig:fig31}). 
For $k\ge 1$, $x_1,x_1'\in\mathbb{Z}$ and $(r,j),(r',j')\in\{0,1\}\times S_0$, let $\hat{g}^{(k)}_{(r,j),(r',j')}(x_1,x_1')$ be the probability that the MA-process $\{\hat{\bY}_n\}$ starting from $(x_1,k,r,j)$ visits a state in $\hat{\mathbb{S}}_0$ for the first time and the state is $(x_1',0,r',j')$, i.e.,
\[
\hat{g}^{(k)}_{(r,j),(r',j')}(x_1,x_1') 
= \mathbb{P}\big( \hat{\bY}_{\tau_0}=(x_1',0,r',j'),\,\tau_0<\infty \,\big|\,\hat{\bY}_0=(x_1,k,r,j) \big).
\]
We denote the matrix of them by $\hat{G}^{(k)}_{x_1,x_1'}$, i.e., $\hat{G}^{(k)}_{x_1,x_1'}=\big(\hat{g}^{(k)}_{(r,j),(r',j')}(x_1,x_1'); (r,j),(r',j')\in\{0,1\}\times S_0\big)$. When $k=1$, we omit the superscript $(k)$ such as $\hat{G}^{(1)}_{x_1,x_1'}=\hat G_{x_1,x_1'}$. 
Since $\{\hat{\bY}_n\}$ is space-homogeneous, we have $\hat{G}^{(k)}_{x_1,x_1'}= \hat{G}^{(k)}_{0,x_1'-x_1}$. By the strong Markov property, for $k\ge 1$ and $x_1\in\mathbb{Z}$, $\hat{\Phi}_{(0,k),(x_1,0)}$ is represented in terms of $\hat{G}^{(k)}_{0,x_1'}$ as 
\begin{equation}
\hat{\Phi}_{(0,k),(x_1,0)} 
= \sum_{x_1'=-\infty}^\infty \hat{G}^{(k)}_{0,x_1'} \hat{\Phi}_{(x_1',0),(x_1,0)}, 
\end{equation}
and this leads us to 
\begin{equation}
\hat{\Phi}_{(0,k),*}(z)
= \hat{G}^{(k)}_{0,*}(z) \hat{\Phi}_{(0,0),*}(z), 
\label{eq:hatPhiG_eq1}
\end{equation}
where 
\begin{equation}
\hat{G}^{(k)}_{0,*}(z) = \sum_{x_1'=-\infty}^\infty z^{x_1'} \hat{G}^{(k)}_{0,x_1'}  
\label{eq:hatG_def0}
\end{equation}
and we use \eqref{eq:hatPhi_eq1}. Since $\{\hat{\bY}_n\}$ is skip free and space-homogeneous, we have by the strong Markov property that
\begin{equation}
\hat{G}^{(k)}_{0,*}(z) = \hat{G}_{0,*}(z)^k. 
\label{eq:Gk_eq1}
\end{equation}
As a result, by \eqref{eq:hatphic2_def}, \eqref{eq:hatPhi_eq1}, \eqref{eq:hatPhiG_eq1} and \eqref{eq:Gk_eq1}, we obtain 
\begin{align}
&\hat{\bvarphi}^{\bc}_2(z) = \sum_{k=1}^\infty\ \sum_{i_2\in\{-1,0,1\}} \hat{\bnu}_{(0,k)} (\hat{A}^{\{2\}}_{*,i_2}(z)-\hat{A}^{\{1,2\}}_{*,i_2}(z))\, \hat{G}_{0,*}(z)^{k+i_2} \hat{\Phi}_{(0,0),*}(z).
\label{eq:hatphic2_eq1}
\end{align}

%
%
We make a preparation for analyzing $\hat{\bvarphi}^{\bc}_2(z)$ through \eqref{eq:hatphic2_eq1}. Define a matrix generating function of transition probability blocks as 
\begin{align*}
&\hat{A}^{\{1,2\}}_{*,*}(z_1,z_2)=\sum_{i_1,i_2\in\{-1,0,1\}} z_1^{i_1} z_2^{i_2} \hat{A}^{\{1,2\}}_{i_1,i_2} .
\end{align*}
Define a domain $\hat{\Gamma}^{\{1,2\}}$ as
\[
\hat{\Gamma}^{\{1,2\}} = \{(\theta_1,\theta_2)\in\mathbb{R}^2; \spr(\hat{A}^{\{1,2\}}_{*,*}(e^{\theta_1},e^{\theta_2})) < 1 \}, 
\]
whose closure is a convex set, and define the extreme values $\hat\theta_1^{min}$ and $\hat\theta_1^{max}$ of $\hat{\Gamma}^{\{1,2\}}$ as 
\[
\hat\theta_1^{min} = \inf\{\theta_1\in\mathbb{R}; (\theta_1,\theta_2)\in\hat{\Gamma}^{\{1,2\}} \},\quad 
\hat\theta_1^{max} = \sup\{\theta_1\in\mathbb{R}; (\theta_1,\theta_2)\in\hat{\Gamma}^{\{1,2\}} \}. 
\]
%
\begin{figure}[tb]
\begin{center}
\includegraphics[width=60mm,trim=0 0 0 0]{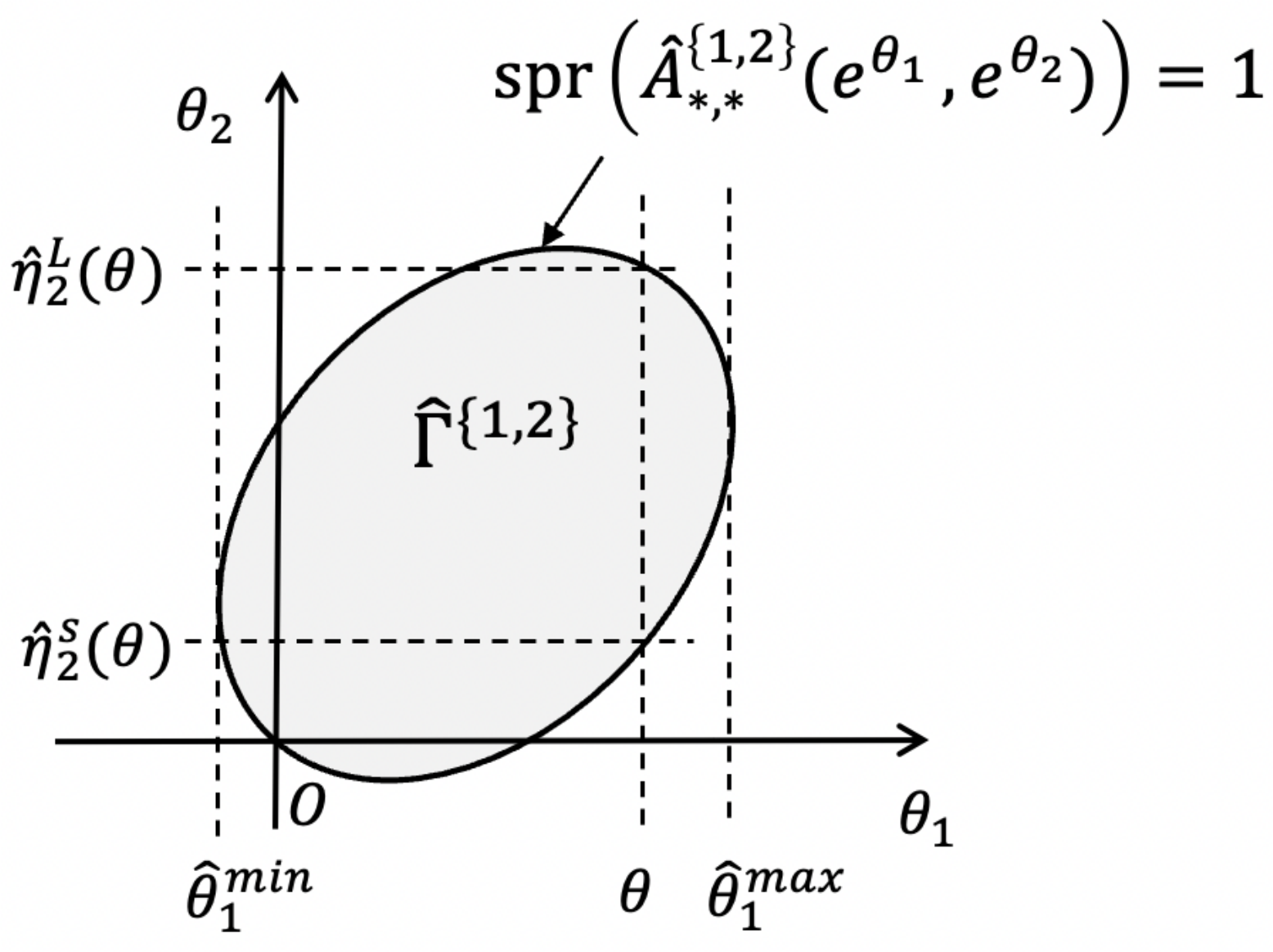} 
\caption{Domain $\hat\Gamma^{\{1,2\}}$}
\label{fig:fig32}
\end{center}
\end{figure}
For $\theta_1\in[\hat\theta_1^{min},\hat\theta_1^{max}]$, let $\hat{\eta}_2^s(\theta_1)$ and $\hat{\eta}_2^L(\theta_1)$ be the two real roots to the following equation:
\[
\spr(\hat{A}^{\{1,2\}}_{*,*}(e^{\theta_1},e^{\theta_2}))=1,
\]
counting multiplicity, where $\hat{\eta}_2^s(\theta_1)\le\hat{\eta}_2^L(\theta_1)$ (see Fig.\ \ref{fig:fig32}). 
The matrix generating function $\hat{G}_{0,*}(z)$ corresponds to a so-called G-matrix in the matrix analytic method of the queueing theory (see, for example, Neuts \cite{Neuts94}).  For $\theta\in[\hat\theta_1^{min},\hat\theta_1^{max}]$, consider the following matrix quadratic equation:
\begin{equation}
\hat{A}^{\{1,2\}}_{*,-1}(e^\theta) + \hat{A}^{\{1,2\}}_{*,0}(e^\theta) X + \hat{A}^{\{1,2\}}_{*,1}(e^\theta) X^2 = X, 
\label{eq:hatG_eq}
\end{equation}
then $\hat{G}_{0,*}(e^\theta)$ is given by the minimum nonnegative solution to the equation, and we have by Lemma 2.5 of Ozawa \cite{Ozawa21} that
\begin{equation}
\spr(\hat{G}_{0,*}(e^\theta))=e^{\hat{\eta}_2^s(\theta)}. 
\end{equation}
Let $\alpha_1(z)$ be the maximum eigenvalue of $\hat{G}_{0,*}(z)$ and $\alpha_i(z),\,i=2,3...,2s_0,$ be other eigenvalues, counting multiplicity. We have, for $\theta\in[\hat\theta_1^{min},\hat\theta_1^{max}]$, $\alpha_1(e^\theta)=\spr(\hat{G}_{0,*}(e^\theta))=e^{\hat\eta_2^s(\theta)}$. 
$\hat{G}_{0,*}(z)$ satisfies the following properties. 
\begin{proposition} \label{pr:hatG_properties}
\begin{itemize}
\item[(i)] $\hat{G}_{0,*}(z)$ is absolutely convergent and entry-wise analytic in the open annual domain $\Delta_{e^{\hat\theta_1^{min}},e^{\hat\theta_1^{max}}}$.  
\item[(ii)] For every $z\in\Delta_{e^{\hat\theta_1^{min}},e^{\hat\theta_1^{max}}}$,  $\spr(\hat{G}_{0,*}(z))\le \spr(\hat{G}_{0,*}(|z|))=e^{\hat{\eta}_2^s(\log |z|)}$. Furthermore, if $z\ne |z|$, then $\spr(\hat{G}_{0,*}(z))<\spr(\hat{G}_{0,*}(|z|))=e^{\hat{\eta}_2^s(\log |z|)}$. 
\item[(iii)] For every $\theta\in[\hat\theta_1^{min},\hat\theta_1^{max}]$, $\alpha_1(e^\theta)> |\alpha_i(e^\theta)|$ for every $i\in\{2,3,...,2s_0\}$. 
\end{itemize}
\end{proposition}

\begin{proof}
Since $\hat{G}_{0,*}(z)$ is given by Laurent series \eqref{eq:hatG_def0} and it is absolutely convergent in the closure of $\Delta_{e^{\hat\theta_1^{min}},e^{\hat\theta_1^{max}}}$, we immediately obtain part (i) of the proposition (see, for example, Section II.1 of Markushevich \cite{Markushevich05}).
By Lemma 4.1 of Ozawa and Kobayashi \cite{Ozawa18}, for every $z\in\Delta_{e^{\hat\theta_1^{min}},e^{\hat\theta_1^{max}}}$, $\spr(\hat{G}_{0,*}(z))\le\spr(|\hat{G}_{0,*}(z)|)\le\spr(\hat{G}_{0,*}(|z|))$, and we obtain the first half of part (ii) of the proposition. The second half is obtained by part (i) of Lemma 4.3 of Ref.\ \cite{Ozawa18}.
Since, under Assumption \ref{as:Y12_onZpmZmp_irreducible}, the lossy Markov chain derived from $\{\hat{\bY}_n\}$ by restricting the state space to $\mathbb{Z}\times\mathbb{Z}_+\times\{0,1\}\times S_0$ is irreducible (see Remark \ref{re:Y12_onZpmZmp_irreducible}), every column of $\hat{G}_{0,*}(e^\theta)$ is positive or zero (see Appendix C of Ozawa \cite{Ozawa21}; a result similar to that holding for rate matrices also holds for G-matrices). 
Hence, nonnegative matrix $\hat{G}_{0,*}(e^\theta)$ has just one primitive class (irreducible and aperiodic class), and this implies part (iii) of the proposition. 
\end{proof}

%
%
We get back to \eqref{eq:hatphic2_eq1} and apply results in Ref.\ \cite{Ozawa18}. 
For $i\in\{1,2\}$, define $\theta_i^*$ and $\theta_i^\dagger$ as
\begin{align}
&\theta_i^* = \sup\{\theta_i\in\mathbb{R}: (\theta_1,\theta_2)\in\Gamma^{\{i\}} \},\quad 
\theta_i^\dagger=\sup\{\theta_i; (\theta_1,\theta_2)\in\Gamma^{\{3-i\}}\cap\Gamma^{\{1,2\}} \}, 
\label{eq:thetai_sd}
\end{align}
then we have $\xi_{(1,0)}=\min\{\theta_1^*,\theta_1^\dagger\}$ and $\xi_{(0,1)}=\min\{\theta_2^*,\theta_2^\dagger\}$; For another equivalent definition of $\theta_i^*$ and $\theta_i^\dagger$ and for the properties of $\xi_{(1,0)}$ and $\xi_{(0,1)}$, see Appendix \ref{sec:block_2dQBD_results}. 
Since $\hat{\bnu}_{(0,k)}=\begin{pmatrix} \bnu_{(0,2k)} & \bnu_{(0,2k+1)} \end{pmatrix}$ for $k\ge 0$, the radius of convergence of the power series of the sequence $\{\hat{\bnu}_{(0,k)}\}_{k\ge 0}$ is given by $e^{2\xi_{(0,1)}}$. Taking this point into account, we define $\hat\theta_1^\dagger$ and $\hat\theta_1^{\dagger,\xi}$ as
\[
\hat\theta_1^\dagger = \max\{\theta\in[\hat\theta_1^{min},\hat\theta_1^{max}]; \hat\eta_2^s(\theta)\le 2\theta_2^* \},\quad 
\hat\theta_1^{\dagger,\xi} = \max\{\theta\in[\hat\theta_1^{min},\hat\theta_1^{max}]; \hat\eta_2^s(\theta)\le 2\xi_{(0,1)} \}.
\]
Since $\xi_{(0,1)}=\min\{\theta_2^*,\theta_2^\dagger\}$, we have $\hat\theta_1^\dagger\ge \hat\theta_1^{\dagger,\xi}$. 
The following is a key proposition for analyzing $\hat{\bvarphi}^{\bc}_2(z)$.
\begin{proposition} \label{pr:hatxi2c_eq}
\begin{itemize}
\item[(i)] We always have $\hat\xi_{\bc,2}\ge \hat\theta_1^{\dagger,\xi}$.
\item[(ii)] If $\hat\theta_1^\dagger<\hat\theta_1^{max}$ and $\theta_2^*<\theta_2^\dagger$, then $\hat\bvarphi^{\bc}_2(z)$ is elementwise analytic in $\Delta_{1,e^{\hat\theta_1^\dagger}}\cup(\partial\Delta_{e^{\hat\theta_1^\dagger}}\setminus\{e^{\hat\theta_1^\dagger}\})$. 
\item[(iii)] If $\hat\theta_1^\dagger<\hat\theta_1^{max}$ and $\theta_2^*<\theta_2^\dagger$, then $\hat\xi_{\bc,2}= \hat\theta_1^\dagger$ and, for some positive vector $\hat\bg^{\bc}_2$, 
\begin{equation}
\lim_{\theta\,\uparrow\,\hat\theta_1^\dagger} (e^{\hat\theta_1^\dagger}-e^{\theta}) \hat\bvarphi^{\bc}_2(e^\theta) = \hat\bg^{\bc}_2.
\label{eq:hatvarphi_2_limit0}
\end{equation}
\end{itemize}
\end{proposition}

%
%
Before proving the proposition, we give one more proposition. 
Let ${}^{(1,2)}\bnu_{(0,k)}, k\in\mathbb{N},$ be the stationary probability vectors in the (1,2)-block state process $\{{}^{(1,2)}\bY_n\}$ derived from the original 2d-QBD process; See Subsection \ref{sec:block_state_process} and Appendix \ref{sec:block_2dQBD_results}. Since $\{{}^{(1,2)}\bY_n\}$ is a 2d-QBD process, we can apply results in Ref.\ \cite{Ozawa18}. 
Define a vector generating function $\hat\bvarphi_2(z)$ as 
\[
\hat\bvarphi_2(z) = \sum_{k=1}^\infty z^k \hat{\bnu}_{(0,k)}, 
\] 
which is identical to ${}^{(1,2)}\bvarphi_2(z) = \sum_{k=1}^\infty z^k\, {}^{(1,2)}{\bnu}_{(0,k)}$ since we have $\hat\bnu_{(0,k)}=\begin{pmatrix} \bnu_{(0,2k)} & \bnu_{(0,2k+1)} \end{pmatrix}={}^{(1,2)}\bnu_{(0,k)}$ for every $k\in\mathbb{N}$. 
If ${}^{(1,2)}\theta_2^*=2 \theta_2^*<{}^{(1,2)}\theta_2^\dagger=2 \theta_2^\dagger$ (for the definitions of ${}^{(1,2)}\theta_2^*$ and ${}^{(1,2)}\theta_2^\dagger$, see Appendix \ref{sec:block_2dQBD_results}), then $\{{}^{(1,2)}\bY_n\}$ is classified into Type I ($\psi_2(\bar z_2^*)>1$) or Type II in the notation of Ref.\ \cite{Ozawa18}, where inequality $\psi_2(\bar z_2^*)>1$ corresponds to ${}^{(1,2)}\theta_2^*<{}^{(1,2)}\theta_2^{max}$. 
In our case, inequality $\theta_2^*<\theta_2^\dagger$ implies this condition since $\theta_2^\dagger\le \theta_2^{max}$ and ${}^{(1,2)}\theta_2^{max}=2 \theta_2^{max}$. Therefore, if $\theta_2^*<\theta_2^\dagger$, we see by Corollary 5.1 of Ref.\ \cite{Ozawa18} that $z=e^{{}^{(1,2)}\theta_2^*}=e^{2\theta_2^*}$ is a pole of ${}^{(1,2)}\bvarphi_2(z)$, and the same property also holds for $\hat\bvarphi_2(z)$. 
Define $\hat{U}_2(z)$ as
\[
\hat{U}_2(z) = \hat{A}^{\{2\}}_{0,*}(z)+\hat{A}^{\{2\}}_{1,*}(z) \hat{G}_{*,0}(z), 
\]
where $\hat{G}_{*,0}(z)$ is the G-matrix generated from the triplet $\{\hat{A}^{\{1,2\}}_{-1,*}(z),\hat{A}^{\{1,2\}}_{0,*}(z),\hat{A}^{\{1,2\}}_{1,*}(z)\}$ (see Subsection 4.1 of Ref.\ \cite{Ozawa18}) and satisfies the following matrix quadratic equation:
\begin{equation}
\hat{A}^{\{1,2\}}_{-1,*}(z) + \hat{A}^{\{1,2\}}_{0,*}(z) X + \hat{A}^{\{1,2\}}_{1,*}(z) X^2 = X. 
\end{equation}
By the definition, $\hat{U}_2(z)$ is identical to ${}^{(1,2)}{U}_2(z)$ of the $(1,2)$-block state process (for the definition of ${}^{(1,2)}{U}_2(z)$, see Appendix \ref{sec:block_2dQBD_results}). Let $\hat\bu_{U}(z)$ and $\hat\bv_{U}(z)$ be the left and right eigenvectors of $\hat{U}_2(z)$ with respect to the maximum eigenvalue of $\hat{U}_2(z)$, satisfying $\hat\bu_{U}(z)\hat\bv_{U}(z)=1$. 
By Corollary 5.1 of Ref.\ \cite{Ozawa18}, considering correspondence between $\hat\bvarphi_2(z)$ and ${}^{(1,2)}\bvarphi_2(z)$, we immediately obtain the following. 
\begin{proposition} \label{pr:hatvarphi2_pole}
If $\theta_2^*<\theta_2^\dagger$, then $\hat\eta_2^s(\hat\theta_1^\dagger)=2\theta_2^*$ and for some positive constant $\hat g_2$, 
\begin{equation}
\lim_{\theta\,\uparrow\,2\theta_2^*} (e^{2\theta_2^*}-e^{\theta}) \hat\bvarphi_2(e^\theta) = \hat g_2 \hat\bu_{U}(e^{2\theta_2^*}), 
\label{eq:hatvarphi12_limit}
\end{equation}
where $\hat\bu_{U}(e^{2\theta_2^*})$ is positive. 
\end{proposition}

Note that, under Assumption \ref{as:Y12_onZpmZmp_irreducible}, the modulus of every eigenvalue of $\hat{G}_{*,0}(e^{2\theta_2^*})$ except for the maximum one is less than $\spr(\hat{G}_{*,0}(e^{2\theta_2^*}))$ (see Proposition \ref{pr:hatG_properties}), and it is not necessary for using Corollary 5.1 of Ref.\ \cite{Ozawa18} to assume all the eigenvalues of $\hat{G}_{*,0}(e^{2\theta_2^*})$ are distinct (i.e., Assumption 2.5 of Ref.\ \cite{Ozawa18}). 
  
%
\begin{proof}[Proof of Proposition \ref{pr:hatxi2c_eq}]
Temporary, define $D(z,w)$ as 
\[
D(z,w) = \hat A^{\{2\}}_{*,-1}(z)+\hat A^{\{2\}}_{*,0}(z) w+\hat A^{\{2\}}_{*,1}(z) w^2- I w, 
\]
where $w$ is a matrix or scalar. By \eqref{eq:hatphic2_eq1} and \eqref{eq:hatG_eq}, $\hat{\bvarphi}^{\bc}_2(z)$ is represented as 
\[
\hat{\bvarphi}^{\bc}_2(z) = \ba(z,\hat{G}_{0,*}(z)) \hat{\Phi}_{(0,0),*}(z), 
\] 
where
\[
\ba(z,w) = \sum_{k=1}^\infty\ \hat{\bnu}_{(0,k)} D(z,\hat{G}_{0,*}(z)) w^{k-1}. 
\]
Later, we will prove that $\hat\theta_1^{min}=\theta_{\bc}^{min}$ and $\hat\theta_1^{max}=\theta_{\bc}^{max}$ (see \eqref{eq:hattheta_eq1}). Hence, by Proposition \ref{pr:Phic_finite} and expression \eqref{eq:hatPhixs_def}, $\hat{\Phi}_{(0,0),*}(z)$ is absolutely convergent in $\Delta_{e^{\hat\theta_1^{min}},e^{\hat\theta^{max}_1}}$.
We, therefore, focus on $\ba(z,\hat{G}_{0,*}(z))$. Let $\hat{G}_{0,*}(z)=V(z) J(z) V(z)^{-1}$ be the Jordan decomposition of $\hat{G}_{0,*}(z)$. Since $\hat{G}_{0,*}(z)^{k-1}=V(z) J(z)^{k-1} V(z)^{-1}$, we have 
\begin{equation}
\ba(z,\hat{G}_{0,*}(z))^\top = (V(z)^{-1})^\top \sum_{k=1}^\infty \left(\hat{\bnu}_{(0,k)} \otimes (J(z)^\top)^{k-1}\right)\, \vec\big((D(z,\hat{G}_{0,*}(z))V(z))^\top\big), 
\label{eq:a_theta_eq}
\end{equation}
where $\otimes$ is the Kronecker product and we use the identity $\vec(ABC)=(C^\top\otimes A)\, \vec(B)$ for matrices $A$, $B$ and $C$ (for  the identity, see Horn and Johnson \cite{Horn91}). 
Define a real value $\theta'_1$ as
\begin{equation}
\theta'_1=\arg\min_{\theta\in[\hat\theta_1^{min},\hat\theta_1^{max}]} \hat\eta_2^s(\theta), 
\label{eq:theta1p}
\end{equation}
then $\hat\eta_2^s(\theta)$ is strictly  increasing in $(\theta'_1,\hat\theta_1^{max})$.  Hence, by part (ii) of Proposition \ref{pr:hatG_properties}, for every $z\in\Delta_{e^{\theta'_1},e^{\hat\theta_1^{\dagger,\xi}}}$, $\spr(\hat{G}_{0,*}(z))\le e^{\hat\eta_2^s(\log |z|)}<e^{\hat\eta_2^s(\hat\theta_1^{\dagger,\xi})}$. 
Since $e^{2\xi_{(0,1)}}$ is the radius of convergence of the power series of the sequence $\{\hat{\bnu}_{(0,k)}\}_{k\ge 0}$, we see that, for every $i\in\{1,2,...,2 s_0\}$, each entry of $\sum_{k=1}^\infty [\hat{\bnu}_{(0,k)}]_i\, (J(z)^\top)^{k-1}$ is absolutely convergent in $z\in\Delta_{e^{\theta'_1},e^{\hat\theta_1^{\dagger,\xi}}}$. As a result, $\hat{\bvarphi}^{\bc}_2(z)$ as well as $\ba(z)$ is absolutely convergent in $\Delta_{e^{\theta'_1},e^{\hat\theta_1^{\dagger,\xi}}}$ and we obtain $\hat\xi_{\bc,2}\ge \hat\theta_1^{\dagger,\xi}$. This completes the proof of part (i) of the proposition. 

%
Next, supposing $\hat\theta_1^\dagger< \hat\theta_1^{max}$, we consider the case where $\theta_2^*<\theta_2^\dagger$. In this case, we have $\hat\eta_2^s(\hat\theta_1^\dagger)= 2\xi_{(0,1)}=2\theta_2^*<2 \theta_2^{max}$ and $\hat\theta_1^\dagger=\hat\theta_1^{\dagger,\xi}$ since $\xi_{(0,1)}=\min\{\theta_2^*,\theta_2^\dagger\}$ and $\theta_2^\dagger\le \theta_2^{max}$. 
We prove part (ii) of the proposition in a manner similar to that used in the proof of Proposition 5.1 of Ref.\ \cite{Ozawa18}, which is given in Ozawa and Kobayashi \cite{Ozawa18b}. 
Let $X=(x_{k,l})$ be an $2 s_0\times 2 s_0$ complex matrix. For $z\in\Delta_{e^{\hat\theta_1^{min}},e^{\hat\theta_1^{max}}}$, if $|w|<e^{\hat\eta_2^s(\hat\theta_1^\dagger)}$, $\ba(z,w)$ is absolutely convergent, and by Lemma 3.2 of Ref.\ \cite{Ozawa18}, we see that if $\spr(X)<e^{\hat\eta_2^s(\hat\theta_1^\dagger)}$, each element of $a(z,X)$ is absolutely convergent. 
This implies that each element of $a(z,X)$ is analytic as a complex function of $4 s_0^2+1$ variables in $\{(z,x_{kl}; k,l=1,2,...,2 s_0)\in\mathbb{C}^{4 s_0^2+1}; e^{\hat\theta_1^{min}}<|z|<e^{\hat\theta_1^{max}}, \spr(X)<e^{\hat\eta_2^s(\hat\theta_1^\dagger)}\}$. 
By parts (i) and (ii) of Proposition \ref{pr:hatG_properties}, for any $z_0\in\Delta_{e^{\theta_1'},e^{\hat\theta_1^\dagger}}\cup(\partial\Delta_{e^{\hat\theta_1^\dagger}}\setminus\{e^{\hat\theta_1^\dagger}\})$, $\hat{G}_{0,*}(z)$ is entry-wise analytic at $z=z_0$ and $\spr(\hat{G}_{0,*}(z_0))<e^{\hat\eta_2^s(\hat\theta_1^\dagger)}$, where $\theta_1'$ is given by \eqref{eq:theta1p}. 
Hence, the composite function $\ba(z,\hat{G}_{0,*}(z))$ is elementwise analytic in $z\in\Delta_{e^{\theta_1'},e^{\hat\theta_1^\dagger}}\cup(\partial\Delta_{e^{\hat\theta_1^\dagger}}\setminus\{e^{\hat\theta_1^\dagger}\})$. 
Under Assumption \ref{as:2dQBD_stable}, since we have $\hat\eta_2^s(\hat\theta_1^\dagger)=2\theta_2^*>0$,  $\hat\eta_2^s(\theta_1')\le 0$ and $\hat\eta_2^s(0)\le 0$, if $\theta_1'>0$ then $\hat\eta_2^s(\theta)\le 0<\hat\eta_2^s(\theta_1^\dagger)$ for every $\theta\in[0,\theta_1']$. Hence, we can replace $e^{\theta_1'}$ with $e^0=1$ and see that $\ba(z,\hat{G}_{0,*}(z))$ is elementwise analytic in $z\in\Delta_{1,e^{\hat\theta_1^\dagger}}\cup(\partial\Delta_{e^{\hat\theta_1^\dagger}}\setminus\{e^{\hat\theta_1^\dagger}\})$. 
By Proposition \ref{pr:Phic_finite} and expression \eqref{eq:hatPhixs_def}, $\hat{\Phi}_{(0,0),*}(z)$ is also entry-wise analytic in the same domain. 
This completes the proof of part (ii) of the proposition. 

%
Finally, supposing $\hat\theta_1^\dagger< \hat\theta_1^{max}$, we consider the case where $\theta_2^*<\theta_2^\dagger$ again. 
By part (iii) of Proposition \ref{pr:hatG_properties}, $\spr(\hat{G}_{0,*}(e^{\hat\theta_1^\dagger}))=e^{\hat\eta_2^s(\hat\theta_1^\dagger)}=e^{2\theta_2^*}$ is a simple eigenvalue of $\hat{G}_{0,*}(e^{\hat\theta_1^\dagger})$, and the modulus of every eigenvalue of $\hat{G}_{0,*}(e^{\hat\theta_1^\dagger})$ except for $e^{2\theta_2^*}$ is less than $e^{2\theta_2^*}$. 
Hence, we have, for every $i\in\{1,2,...,2 s_0\}$, 
\[
\lim_{\theta\,\uparrow\,\hat\theta_1^\dagger} (e^{\hat\theta_1^\dagger}-e^{\theta}) \sum_{k=1}^\infty\, [\hat{\bnu}_{(0,k)}]_i\, (J(e^{\theta})^\top)^{k-1} 
= \lim_{\theta\,\uparrow\,\hat\theta_1^\dagger} (e^{\hat\theta_1^\dagger}-e^{\theta}) [\hat{\bvarphi}_2(e^{\hat\eta_2^s(\theta)})]_i\, e^{-\hat\eta_2^s(\theta)}\, \diag \!\begin{pmatrix} 1 & 0 & \cdots & 0 \end{pmatrix}, 
\]
where we assume $[J(e^\theta)]_{1,1}=\alpha_1(e^\theta)=e^{\hat\eta_2^s(\theta)}$ without loss of generality.  By \eqref{eq:a_theta_eq}, this leads us to 
\begin{align}
\lim_{\theta\,\uparrow\,\hat\theta_1^\dagger} (e^{\hat\theta_1^\dagger}-e^{\theta}) \ba(e^\theta,\hat{G}_{0,*}(e^\theta))
&= \lim_{\theta\,\uparrow\,\hat\theta_1^\dagger} (e^{\hat\theta_1^\dagger}-e^{\theta})\,\hat{\bvarphi}_2(e^{\hat\eta_2^s(\theta)})\, e^{-\hat\eta_2^s(\hat\theta_1^\dagger)} D(e^{\hat\theta_1^\dagger},e^{\hat\eta_2^s(\hat\theta_1^\dagger)}) \hat\bv_G(e^{\hat\theta_1^\dagger}) \hat\bu_G(e^{\hat\theta_1^\dagger}), 
\end{align}
where $\hat\bu_G(e^\theta)$ and $\hat\bv_G(e^\theta)$ are the left and right eigenvectors of $\hat{G}_{0,*}(e^\theta)$ with respect to the eigenvalue $e^{\hat\eta_2^s(\theta)}$, satisfying $\hat\bu_G(e^\theta) \hat\bv_G(e^\theta) =1$.
By Proposition \ref{pr:hatvarphi2_pole}, we have
\begin{align}
\lim_{\theta\,\uparrow\,\hat\theta_1^\dagger} (e^{\hat\theta_1^\dagger}-e^{\theta})\,\hat{\bvarphi}_2(e^{\hat\eta_2^s(\theta)})
&= \lim_{\theta\,\uparrow\,\hat\theta_1^\dagger} \frac{e^{\hat\theta_1^\dagger}-e^{\theta}}{e^{\hat\eta_2^s(\hat\theta_1^\dagger)}-e^{\hat\eta_2^s(\theta)}} (e^{\hat\eta_2^s(\hat\theta_1^\dagger)}-e^{\hat\eta_2^s(\theta)}) \hat{\bvarphi}_2(e^{\hat\eta_2^s(\theta)}) \cr
&= \hat g'_2 \hat\bu_U(e^{\hat\eta_2^s(\hat\theta_1^\dagger)}), 
\end{align}
where $\hat g'_2=\hat g_2 e^{\hat\theta_1^\dagger-\hat\eta_2^s(\hat\theta_1^\dagger)}/\hat\eta_{2,\theta}^s(\hat\theta_1^\dagger)$, $\hat\eta_{2,\theta}^s(x)=\frac{d}{dx}\hat\eta_{2}^s(x)$ and $\hat g_2$ is a positive constant. Since $\hat\eta_2^s(\theta)$ is strictly  increasing in $(\theta'_1,\hat\theta_1^{max})$, we have $\hat\eta_{2,\theta}^s(\hat\theta_1^\dagger)>0$, and this implies $\hat g'_2>0$.
As a result, we obtain 
\begin{equation}
 \lim_{\theta\,\uparrow\,\hat\theta_1^\dagger} (e^{\hat\theta_1^\dagger}-e^{\theta}) \hat\bvarphi^{\bc}_2(\theta) 
= \hat g'_2 e^{-\hat\eta_2^s(\hat\theta_1^\dagger)} \hat\bu_U(e^{\hat\eta_2^s(\hat\theta_1^\dagger)}) D(e^{\hat\theta_1^\dagger},e^{\hat\eta_2^s(\hat\theta_1^\dagger)}) \hat\bv_G(e^{\hat\theta_1^\dagger}) \hat\bu_G(e^{\hat\theta_1^\dagger}) \hat{\Phi}_{(0,0),*}(\hat\theta_1^\dagger).
\label{eq:atheta_limit}
\end{equation}
In a manner similar to that used in the proof of Lemma 5.5 (part (1)) of Ref.\ \cite{Ozawa18}, it can be seen that $\hat\bu_U(e^{\hat\eta_2^s(\hat\theta_1^\dagger)}) D(e^{\hat\theta_1^\dagger},e^{\hat\eta_2^s(\hat\theta_1^\dagger)}) \hat\bv_G(e^{\hat\theta_1^\dagger}) >0$. 
Since $P^{\{1,2\}}$ is irreducible, $\hat{\Phi}_{(0,0),*}(\hat\theta_1^\dagger)$ is positive, and it implies that $\hat\bu_G(\hat\theta_1^\dagger) \hat{\Phi}_{(0,0),*}(\hat\theta_1^\dagger)$ is also positive. 
This completes the proof of part (iii) of the proposition. 
\end{proof}

\begin{remark} \label{re:hatGc_eigen}
In Ref.\ \cite{Ozawa18}, the matrix corresponding to $\hat{G}_{0,*}(\hat\theta_1^\dagger)$ is assumed to have distinct eigenvalues, but that assumption is not necessary in our case. 
In the proof of Proposition \ref{pr:hatxi2c_eq}, the condition required for $\hat{G}_{0,*}(e^\theta)$ is that when $\theta=\hat\theta_1^\dagger$, the maximum eigenvalue $\alpha_1(e^{\hat\theta_1^\dagger})$ is simple and satisfies $\alpha_1(e^{\hat\theta_1^\dagger})> |\alpha_i(e^{\hat\theta_1^\dagger})|$ for every $i\in\{2,3,...,2s_0\}$. As a condition ensuring this point, we have adopted Assumption \ref{as:Y12_onZpmZmp_irreducible}. Under the assumption, the same property also holds for every direction vector in $\mathbb{N}^2$, see the following subsection. 
\end{remark}

%
%
Proposition \ref{pr:hatxi2c_eq} is represented in terms of the parameters given based on the MA-process $\{\hat\bY_n\}$ such as $\hat\theta_1^{max}$ and $\hat\theta_1^\dagger$. We redefined those parameters so that they are given based on the induced MA-process $\{\bY^{\{1,2\}}_n\}$. 
Define a matrix generating function $B(z_1,z_2)$ as
\begin{align*}
B(z_1,z_2)
&= [\hat{A}^{\{1,2\}}_{*,-1}(z_1)]_{0,0}\, z_2^{-2} + [\hat{A}^{\{1,2\}}_{*,-1}(z_1)]_{0,1}\, z_2^{-1} \cr
&\qquad + [\hat{A}^{\{1,2\}}_{*,0}(z_1)]_{0,0} + [\hat{A}^{\{1,2\}}_{*,0}(z_1)]_{0,1}\, z_2 + [\hat{A}^{\{1,2\}}_{*,1}(z_1)]_{0,0}\, z_2^2 \cr
&= A^{\{1,2\}}_{1,-1} z_1 z_2^{-2} + A^{\{1,2\}}_{1,0} z_1 z_2^{-1} + A^{\{1,2\}}_{0,-1} z_2^{-1} + A^{\{1,2\}}_{-1,-1} z_1^{-1} + A^{\{1,2\}}_{0,0} + A^{\{1,2\}}_{1,1} z_1 \cr
&\qquad + A^{\{1,2\}}_{-1,0} z_1^{-1} z_2 + A^{\{1,2\}}_{0,1} z_2 + A^{\{1,2\}}_{-1,1} z_1^{-1} z_2, 
\end{align*}
where, for a block matrix $A$, we denote by $[A]_{i,j}$ the $(i,j)$-block of $A$. This matrix function satisfies
\begin{equation}
B(e^{\theta_1},e^{\theta_2}) = A^{\{1,2\}}_{*,*}(e^{\theta_1-\theta_2},e^{\theta_2}).
\end{equation}
By Remark 2.4 of Ozawa \cite{Ozawa21}, we have
\begin{equation}
\spr(\hat{A}^{\{1,2\}}_{*,*}(e^{\theta_1},e^{\theta_2})) = \spr(B(e^{\theta_1},e^{\theta_2/2})), 
\end{equation}
and this leads us to
\begin{equation}
\spr(\hat{A}^{\{1,2\}}_{*,*}(e^{\theta_1},e^{\theta_2})) = \spr(A^{\{1,2\}}_{*,*}(e^{\theta_1-\theta_2/2},e^{\theta_2/2})). 
\end{equation}
%
\begin{figure}[t]
\begin{center}
\includegraphics[width=90mm,trim=0 0 0 0]{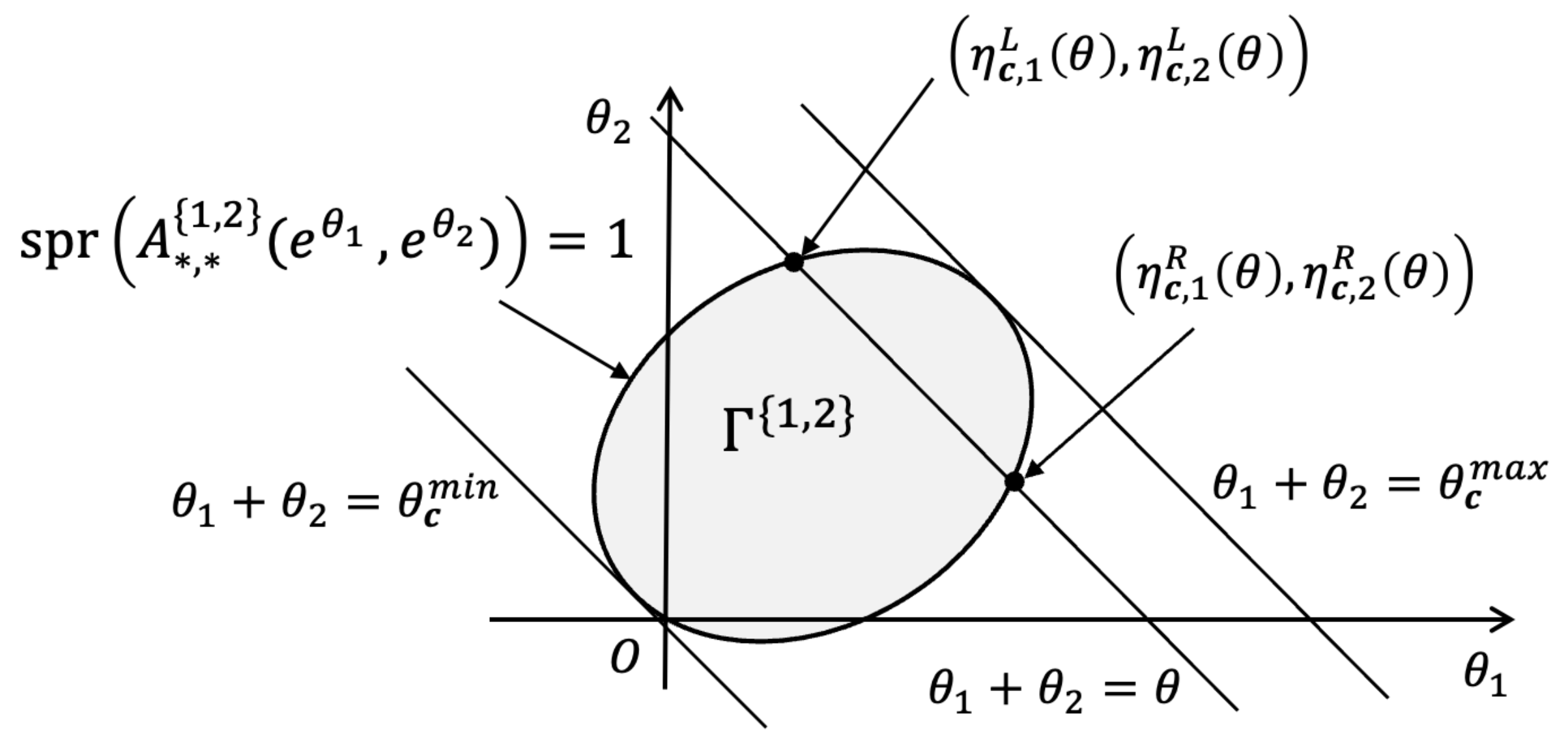} 
\caption{Domain $\Gamma^{\{1,2\}}$}
\label{fig:fig33}
\end{center}
\end{figure}
%
For $\theta\in[\theta_{\bc}^{min},\theta_{\bc}^{max} ]$, let $(\eta^R_{\bc,1}(\theta),\eta^R_{\bc,2}(\theta))$ and $(\eta^L_{\bc,1}(\theta), \eta^L_{\bc,2}(\theta))$ be the two real roots of the simultaneous equations: 
\begin{equation}
\spr(A^{\{1,2\}}_{*,*}(e^{\theta_1},e^{\theta_2}))=1,\quad \theta_1+\theta_2=\theta, 
\label{eq:simuleq_11}
\end{equation}
counting multiplicity, where $\eta^L_{\bc,1}(\theta)\le \eta^R_{\bc,1}(\theta)$ and $\eta^L_{\bc,2}(\theta))\ge \eta^R_{\bc,2}(\theta)$ (see Fig.\ \ref{fig:fig33}).
Since equation $\spr(\hat{A}^{\{1,2\}}_{*,*}(e^{\theta_1},e^{\theta_2}))=1$ is equivalent to $\spr(A^{\{1,2\}}_{*,*}(e^{\theta_1-\theta_2/2},e^{\theta_2/2}))=1$, we have 
\begin{equation}
\hat\theta_1^{min} = \theta_{\bc}^{min},\quad 
\hat\theta_1^{max} = \theta_{\bc}^{max},\quad 
\hat{\eta}_2^s(\theta_1)=2 \eta^R_{\bc,2}(\theta_1), 
\label{eq:hattheta_eq1}
\end{equation}
and  $\hat\theta_1^\dagger$ and $\hat\theta_1^{\dagger,\xi}$ are given by
\begin{align}
&\hat\theta_1^\dagger = \max\{\theta\in[\theta_{\bc}^{min},\theta_{\bc}^{max}]; \eta^R_{\bc,2}(\theta)\le \theta_2^* \}, 
\label{eq:hattheta1dagger_eq} \\
&\hat\theta_1^{\dagger,\xi} = \max\{\theta\in[\theta_{\bc}^{min},\theta_{\bc}^{max}]; \eta^R_{\bc,2}(\theta)\le \xi_{(0,1)} \}.
\label{eq:hattheta1daggerxi_eq}
\end{align}
Hereafter, we denote $\hat\theta_1^\dagger$ and $\hat\theta_1^{\dagger,\xi}$ by $\theta_{\bc,2}^\dagger$ and $\theta_{\bc,2}^{\dagger,\xi}$, respectively, and use \eqref{eq:hattheta1dagger_eq} and \eqref{eq:hattheta1daggerxi_eq} as their definitions. Note that, for $\theta_{\bc,2}^\dagger$ and $\theta_{\bc,2}^{\dagger,\xi}$, we use subscript ``2" instead of ``1" since they are defined by using $\theta_2^*$ and $\xi_{(0,1)}$. 
$\theta_{\bc,2}^\dagger$ has already been defined in Section \ref{sec:intro}. Here we redefine it for the case of $\bc=(1,1)$. After, we also redefine $\theta_{\bc,1}^\dagger$. 
In terms of these parameters, we rewrite Proposition \ref{pr:hatxi2c_eq} as follows. 
\begin{corollary} \label{co:varphic2_pole}
\begin{itemize}
\item[(i)] We always have $\hat\xi_{\bc,2}\ge \theta_{\bc,2}^{\dagger,\xi}$.
\item[(ii)] If $\theta_{\bc,2}^\dagger<\theta_{\bc}^{max}$ and $\theta_2^*<\theta_2^\dagger$, then $\hat\bvarphi^{\bc}_2(z)$ is elementwise analytic in $\Delta_{1,e^{\theta_{\bc,2}^\dagger}}\cup(\partial\Delta_{e^{\theta_{\bc,2}^\dagger}}\setminus\{e^{\theta_{\bc,2}^\dagger}\})$. 
\item[(iii)] If $\theta_{\bc,2}^\dagger<\theta_{\bc}^{max}$ and $\theta_2^*<\theta_2^\dagger$, then $\hat\xi_{\bc,2}= \theta_{\bc,2}^\dagger$ and, for some positive vector $\hat\bg^{\bc}_2$, 
\begin{equation}
\lim_{\theta\,\uparrow\,\theta_{\bc,2}^\dagger} (e^{\theta_{\bc,2}^\dagger}-e^{\theta}) \hat\bvarphi^{\bc}_2(e^\theta) = \hat\bg^{\bc}_2.
\label{eq:hatvarphi_2_limit1}
\end{equation}
\end{itemize}
\end{corollary}

%
%
Define $\hat\bvarphi^{\bc}_1(z)=\begin{pmatrix}\hat\bvarphi^{\bc}_{1,1}(z) & \hat\bvarphi^{\bc}_{1,2}(z)\end{pmatrix}$  and $\hat\xi_{\bc,1}$ analogously to $\hat\bvarphi^{\bc}_2(z)=\begin{pmatrix}\hat\bvarphi^{\bc}_{2,1}(z) & \hat\bvarphi^{\bc}_{2,2}(z)\end{pmatrix}$  and $\hat\xi_{\bc,2}$, respectively. Then, we have $\xi_{\bc,1} \ge \hat\xi_{\bc,1}$.
Define $\theta_{\bc,1}^\dagger$ and $\theta_{\bc,1}^{\dagger,\xi}$ as 
\begin{align}
&\theta_{\bc,1}^\dagger = \max\{\theta\in[\theta_{\bc}^{min},\theta_{\bc}^{max}]; \eta^L_{\bc,1}(\theta)\le \theta_1^* \}, 
\label{eq:hattheta1dagger2_eq} \\
&\theta_{\bc,1}^{\dagger,\xi} = \max\{\theta\in[\theta_{\bc}^{min},\theta_{\bc}^{max}]; \eta^L_{\bc,1}(\theta)\le \xi_{(1,0)} \}.
\label{eq:hattheta2dagger_eq}
\end{align}
With respect to $\hat\bvarphi^{\bc}_1(e^\theta)$, interchangeing the $x_1$-axis with the $x_2$-axis, we immediately obtain by Corollary \ref{co:varphic2_pole} the following. 
\begin{corollary} \label{co:varphic1_pole}
\begin{itemize}
\item[(i)] We always have $\hat\xi_{\bc,1}\ge \theta_{\bc,1}^{\dagger,\xi}$.
\item[(ii)] If $\theta_{\bc,1}^\dagger<\theta_{\bc}^{max}$ and $\theta_1^*<\theta_1^\dagger$, then $\hat\bvarphi^{\bc}_1(z)$ is elementwise analytic in $\Delta_{1,e^{\theta_{\bc,1}^\dagger}}\cup(\partial\Delta_{e^{\theta_{\bc,1}^\dagger}}\setminus\{e^{\theta_{\bc,1}^\dagger}\})$. 
\item[(iii)] If $\theta_{\bc,1}^\dagger<\theta_{\bc}^{max}$ and $\theta_1^*<\theta_1^\dagger$, then $\hat\xi_{\bc,1}= \theta_{\bc,1}^\dagger$ and, for some positive vector $\hat\bg^{\bc}_1$, 
\begin{equation}
\lim_{\theta\,\uparrow\,\theta_{\bc,1}^\dagger} (e^{\theta_{\bc,1}^\dagger}-e^{\theta}) \hat\bvarphi^{\bc}_1(e^\theta) = \hat\bg^{\bc}_1.
\label{eq:hatvarphi_1_limit1}
\end{equation}
\end{itemize}
\end{corollary}

By Proposition \ref{pr:varphic0_finite} and Corollaries \ref{co:varphic2_pole} and \ref{co:varphic1_pole}, we obtain a main result of this subsection as follows.
\begin{theorem} \label{th:xi11_eq}
We have $\xi_{\bc}=\xi_{(1,1)} = \min\{\theta_{\bc,1}^\dagger,\,\theta_{\bc,2}^\dagger\}$, and if $\xi_{\bc}<\theta_{\bc}^{max}$,  the sequence $\{\bnu_{(k,k)}\}_{k\ge 0}$ geometrically decays with ratio $e^{-\xi_{\bc}}$ as $k$ tends to infinity, i.e., for some positive vector $\bg$, 
\[
\bnu_{(k,k)} \sim \bg e^{-\xi_{\bc}k}\ \mbox{as $k\to\infty$}.
\]
\end{theorem}
%
\begin{figure}[t]
\begin{center}
\includegraphics[width=70mm,trim=0 0 0 0]{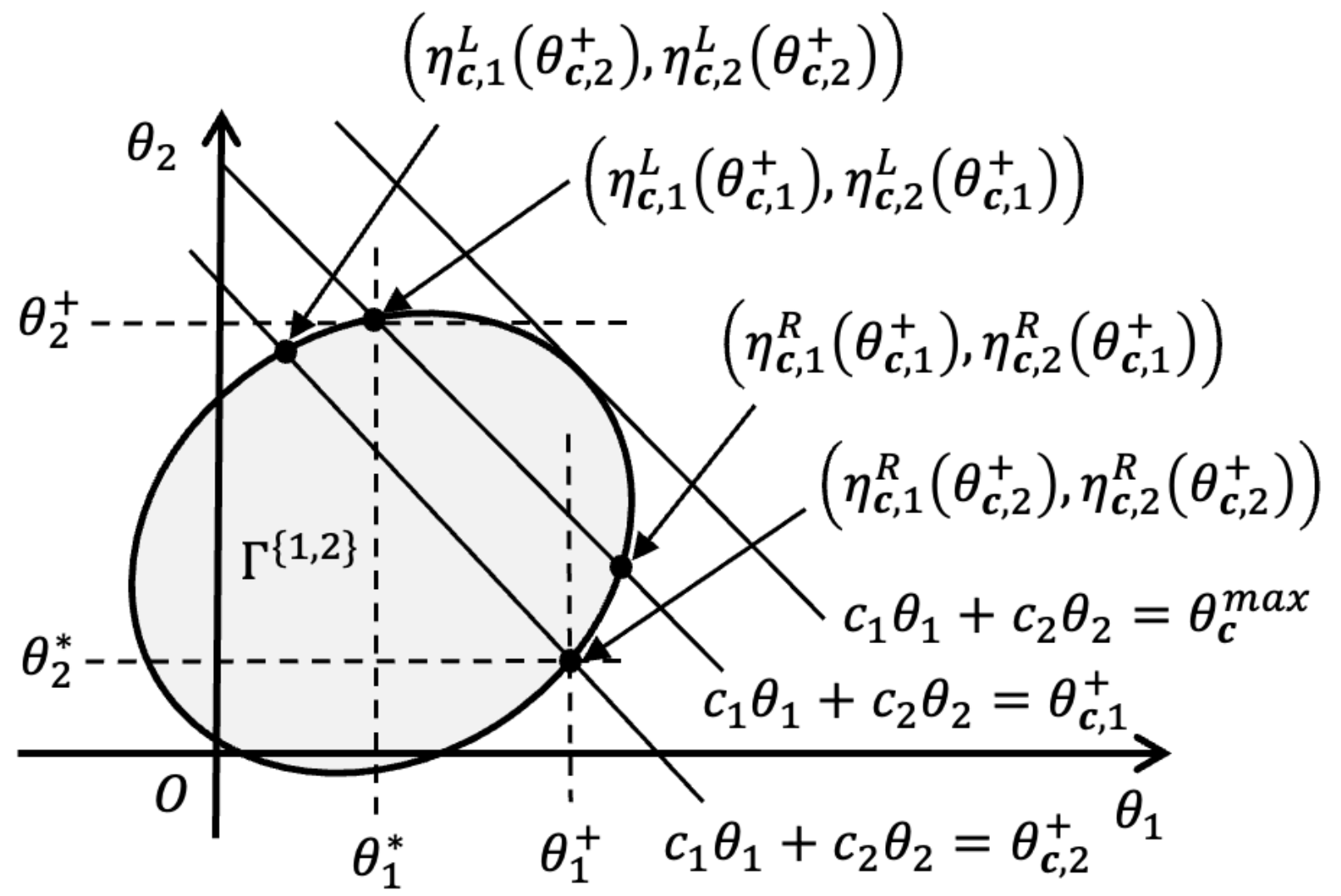} 
\caption{Domain $\Gamma^{\{1,2\}}$}
\label{fig:fig34}
\end{center}
\end{figure}
\begin{proof}
Recall that $\bvarphi^{\bc}(z)=\bvarphi^{\bc}_0(z)+\bvarphi^{\bc}_1(z)+\bvarphi^{\bc}_2(z)$. This $\bvarphi^{\bc}(z)$ is absolutely convergent and elementwise analytic in $\Delta_{e^{\underline{\xi}_{\bc}}}$ since $e^{\underline{\xi}_{\bc}}$ is the radius of the convergence of $\bvarphi^{\bc}(z)$. With respect to the values of $\theta_{\bc,1}^\dagger$ and $\theta_{\bc,2}^\dagger$, we consider the following cases. 

(1) $\theta_{\bc,1}^\dagger=\theta_{\bc,2}^\dagger=\theta_{\bc}^{max}$.\quad 
By Proposition \ref{pr:varphic0_finite} and Corollaries \ref{co:varphic2_pole} and \ref{co:varphic1_pole}, we have
\[
\underline{\xi}_{\bc}
\ge\min\{\xi_{\bc,0}, \xi_{\bc,1}, \xi_{\bc,2}\} 
\ge \min\{\xi_{\bc,0},\hat{\xi}_{\bc,1}, \hat{\xi}_{\bc,2}\}
\ge \min\{\theta_{\bc,1}^{\dagger,\xi}, \theta_{\bc,2}^{\dagger,\xi}\} 
\ge \min\{\theta_{\bc,1}^\dagger, \theta_{\bc,2}^\dagger\}
=\theta_{\bc}^{max}.  
\]
By \eqref{eq:xic_upper}, we have $\bar\xi_{\bc}\le \theta_{\bc}^{max}$, and hence $\xi_{\bc}= \theta_{\bc}^{max}=\min\{\theta_{\bc,1}^\dagger, \theta_{\bc,2}^\dagger\}$. 

(2) $\theta_{\bc,2}^\dagger<\theta_{\bc,1}^\dagger\le \theta_{\bc}^{max}$.\quad 
By Proposition \ref{pr:varphic0_finite} and Corollary \ref{co:varphic1_pole}, $\xi_{c,0}\ge \theta_{\bc}^{max}>\theta_{\bc,2}^\dagger$ and $\xi_{c,1}\ge \theta_{\bc,1}^\dagger> \theta_{\bc,2}^\dagger$. 
We have $\theta_1^*\ge \eta^L_{\bc,1}(\theta_{\bc,1}^\dagger)>\eta^L_{\bc,1}(\theta_{\bc,2}^\dagger)$ and this implies that $\theta_2^*=\eta^R_{\bc,2}(\theta_{\bc,2}^\dagger)< \eta^L_{\bc,1}(\theta_{\bc,2}^\dagger) \le \theta_2^\dagger$ (see Fig.\ \ref{fig:fig34}, where we assume $\bc=(1,1)$). 
Hence, by part (iii) of Corollary \ref{co:varphic2_pole}, $\hat\bvarphi_2^{\bc}(z)$ elementwise diverges at $z=e^{\theta_{\bc,2}^\dagger}$, and we have $\underline{\xi}_{\bc}=\theta_{\bc,2}^\dagger$. 
Since $\underline{\xi}_{\bc}<\theta_{\bc}^{max}\le \xi_{\bc,0}$ and $\underline{\xi}_{\bc}<\theta_{\bc,1}^\dagger\le\hat\xi_{\bc,1}$, $\bvarphi^{\bc}_0(z)$ and $\bvarphi^{\bc}_1(z)$ as well as $\hat\bvarphi^{\bc}_1(z)$ are elementwise analytic on $\partial\Delta_{e^{\underline{\xi}_{\bc}}}$. 
By part (ii) of Corollary \ref{co:varphic2_pole}, $\bvarphi^{\bc}_2(z)$ as well as $\hat\bvarphi^{\bc}_2(z)$ is elementwise analytic on $\partial\Delta_{e^{\underline{\xi}_{\bc}}}\setminus\{e^{\underline{\xi}_{\bc}}\}$. Hence, $\bvarphi^{\bc}(z)$ is elementwise analytic in $\Delta_{e^{\underline{\xi}_{\bc}}}\cup(\partial\Delta_{e^{\underline{\xi}_{\bc}}}\setminus\{e^{\underline{\xi}_{\bc}}\})$. 
As a result, by part (iii) of Corollary \ref{co:varphic2_pole} and Theorem VI.4 of Flajolet and Sedgewick \cite{Flajolet09}, the sequence $\{\bnu_{(k,k)}\}_{k\ge 0}$ geometrically decays with ratio $e^{-\theta^\dagger_{\bc,2}}$ as $k$ tends to infinity and we obtain $\xi_{\bc}=\underline{\xi}_{\bc}= \theta_{\bc,2}^\dagger=\min\{\theta_{\bc,1}^\dagger, \theta_{\bc,2}^\dagger\}<\theta_{\bc}^{max}$. 

(3) $\theta_{\bc,1}^\dagger<\theta_{\bc,2}^\dagger\le \theta_{\bc}^{max}$.\quad 
This case is symmetrical to the previous case. 

(4) $\theta_{\bc,1}^\dagger=\theta_{\bc,2}^\dagger<\theta_{\bc}^{max}$.\quad 
Set $\theta=\theta_{\bc,1}^\dagger$ ($=\theta_{\bc,2}^\dagger$). By Proposition \ref{pr:varphic0_finite}, $\underline{\xi}_{c,0}\ge \theta_{\bc}^{max}>\theta$. 
We have $\theta_1^*=\eta^L_{\bc,1}(\theta)< \eta^R_{\bc,1}(\theta) \le \theta_1^\dagger$ and $\theta_2^*=\eta^R_{\bc,2}(\theta)< \eta^L_{\bc,2}(\theta) \le \theta_2^\dagger$. Hence, in a manner similar to that used in part (2) above, we see that the sequence $\{\bnu_{(k,k)}\}_{k\ge 0}$ geometrically decays with ratio $e^{-\theta}$ as $k$ tends to infinity and obtain $\xi_{\bc}= \theta=\min\{\theta_{\bc,1}^\dagger, \theta_{\bc,2}^\dagger\}<\theta_{\bc}^{max}$. 
\end{proof}

%
%
%
\subsection{In the case of general direction vector $\bc$}

Letting $\bc=(c_1,c_2)\in\mathbb{N}^2$ be a direction vector, we obtain the asymptotic rate $\xi_{\bc}$. For the purpose, we consider the $\bc$-block state process derived from the original 2d-QBD process, $\{{}^{\bc}\bY_n\}=\{( ({}^{\bc}X_{1,n},{}^{\bc}X_{2,n}),({}^{\bc}M_{1,n},{}^{\bc}M_{2,n},{}^{\bc}\!J_n) ) \}$, whose state space is $\mathbb{Z}_+^2\times\mathbb{Z}_{0,c_1-1}\times\mathbb{Z}_{0,c_2-1}\times S_0$. 
Since the state $(k,k,0,0,j)$ of $\{{}^{\bc}\bY_n\}$ corresponds to the state $(c_1 k,c_2 k,j)$ of the original 2d-QBD process, we have for any $j\in S_0$ that  
\begin{equation}
\xi_{\bc} = {}^{\bc}\xi_{(1,1)} =  - \lim_{k\to\infty} \frac{1}{k} \log {}^{\bc}\nu_{(k,k,0,0,j)}, 
\end{equation}
where $\big({}^{\bc}\nu_{(x_1,x_2,r_1,r_2,j)}; (x_1,x_2,r_1,r_2,j)\in \mathbb{Z}_+^2\times\mathbb{Z}_{0,c_1-1}\times\mathbb{Z}_{0,c_2-1}\times S_0 \big)$ is the stationary distribution of $\{{}^{\bc}\bY_n\}$. Therefore, applying the results of the previous subsection to $\{{}^{\bc}\bY_n\}$, we can obtain $\xi_{\bc}$. 

Denote by ${}^{\bc}\!A_{*,*}^{\{1,2\}}(z_1,z_2)$ the matrix generating function of the transition probability blocks of $\{{}^{\bc}\bY_n\}$, corresponding to $A_{*,*}^{\{1,2\}}(z_1,z_2)$ of the original 2d-QBD process (see Appendix \ref{sec:block_2dQBD_results}).
The simultaneous equations corresponding to \eqref{eq:simuleq_11} are given by
\begin{equation}
\spr({}^{\bc}\!A^{\{1,2\}}_{*,*}(e^{\theta_1},e^{\theta_2}))=1,\quad \theta_1+\theta_2=\theta. 
\label{eq:simuleq0_c1c2}
\end{equation}
Since we have by Proposition 4.2 of Ozawa \cite{Ozawa21}  that 
\begin{equation}
\spr({}^{\bc}\!A_{*,*}^{\{1,2\}}(e^{c_1\theta_1},e^{c_2\theta_2})) = \spr(A_{*,*}^{\{1,2\}}(e^{\theta_1},e^{\theta_2})), 
\label{eq:cAssAss_relation}
\end{equation}
simultaneous equations \eqref{eq:simuleq0_c1c2} are equivalent to 
\begin{equation}
\spr(A^{\{1,2\}}_{*,*}(e^{\theta_1},e^{\theta_2}))=1,\quad c_1  \theta_1+c_2 \theta_2=\theta. 
\label{eq:simuleq1_c1c2}
\end{equation}
For $\theta\in[\theta_{\bc}^{min},\theta_{\bc}^{max} ]$, let $(\eta^R_{\bc,1}(\theta),\eta^R_{\bc,2}(\theta))$ and $(\eta^L_{\bc,1}(\theta), \eta^L_{\bc,2}(\theta))$ be the two real roots of simultaneous equations \eqref{eq:simuleq1_c1c2}, counting multiplicity, where $\eta^L_{\bc,1}(\theta)\le \eta^R_{\bc,1}(\theta)$ and $\eta^L_{\bc,2}(\theta)\ge \eta^R_{\bc,2}(\theta)$.
Redefine real values $\theta_{\bc,1}^\dagger$ and $\theta_{\bc,2}^\dagger$ as
\begin{align}
&\theta_{\bc,1}^\dagger = \max\{\theta\in[\theta_{\bc}^{min},\theta_{\bc}^{max}]; \eta^L_{\bc,1}(\theta)\le \theta_1^* \},
\label{eq:thetac1dagger_eq} \\
&\theta_{\bc,2}^\dagger = \max\{\theta\in[\theta_{\bc}^{min},\theta_{\bc}^{max}]; \eta^R_{\bc,2}(\theta)\le \theta_2^* \}, 
\label{eq:thetac2dagger_eq}
\end{align}
which are equivalent to definitions \eqref{eq:theta_bcd0}. 
Since the block state process $\{{}^{\bc}\bY_n\}$ is derived from the original 2d-QBD process, the former inherits all assumptions for the latter, including Assumption \ref{as:Y12_onZpmZmp_irreducible}. Hence, by Theorem \ref{th:xi11_eq}, we immediately obtain the following.
\begin{theorem} \label{th:xic1c2_eq}
For any direction vector $\bc\in\mathbb{N}^2$, $\xi_{\bc}= \min\{\theta_{\bc,1}^\dagger,\,\theta_{\bc,2}^\dagger\}$, and if $\xi_{\bc}<\theta_{\bc}^{max}$,  the sequence $\{\bnu_{k\bc}\}_{k\ge 0}$ geometrically decays with ratio $e^{-\xi_{\bc}}$ as $k$ tends to infinity, i.e., for some constant vector $\bg$, 
\[
\bnu_{k\bc} \sim \bg e^{-\xi_{\bc}k}\ \mbox{as $k\to\infty$}.
\]
\end{theorem}

%
%
\section{Geometric property and an example} \label{sec:discussion}

\subsection{The value of the asymptotic rare $\xi_{\bc}$}

Geometric consideration (see, for example, Miyazawa \cite{Miyazawa11}) is also useful in our case. Here we reconsider Theorem \ref{th:xic1c2_eq} geometrically.
Define two points $\mathrm{Q}_1$ and $\mathrm{Q}_2$ as $\mathrm{Q}_1=(\theta_1^*,\bar{\eta}_2(\theta_1^*))$ and $\mathrm{Q}_2=(\bar{\eta}_1(\theta_2^*),\theta_2^*)$, respectively. For the definition of $\theta_1^*$ and $\theta_2^*$, see \eqref{eq:thetai_sd}, and for the definition of $\bar{\eta}_1(\theta)$ and $\bar{\eta}_2(\theta)$, see Appendix \ref{sec:block_2dQBD_results}. 
Using these points, we define the following classification (see Fig.\ \ref{fig:classification}). 
\begin{itemize}
\item[] Type 1: $\theta_1^* \ge \bar{\eta}_1(\theta_2^*)$ and $\bar{\eta}_2(\theta_1^*) \le \theta_2^*$, \\
Type 2: $\theta_1^* < \bar{\eta}_1(\theta_2^*)$ and $\bar{\eta}_2(\theta_1^*) > \theta_2^*$, \\
Type 3: $\theta_1^* \ge \bar{\eta}_1(\theta_2^*)$ and $\bar{\eta}_2(\theta_1^*) > \theta_2^*$, \\
Type 4: $\theta_1^* < \bar{\eta}_1(\theta_2^*)$ and $\bar{\eta}_2(\theta_1^*) \le \theta_2^*$.
\end{itemize}
Let $\bc=(c_1,c_2)\in\mathbb{N}^2$ be an arbitrary direction vector. For $i\in\{1,2\}$, $\Gamma^{\{i\}}$ satisfies $\Gamma^{\{i\}}=\{(\theta_1,\theta_2)\in\mathbb{R}^2; \theta_i<\theta_i^*\}$. Hence, by \eqref{eq:theta_bcd0}, we have, for $i\in\{1,2\}$, 
\begin{equation}
\theta_{\bc,i}^\dagger = \sup\{c_1\theta_1+c_2\theta_2; (\theta_1,\theta_2)\in\Gamma^{\{1,2\}},\ \theta_{3-i}<\theta_{3-i}^* \}.
\end{equation}
From this representation for $\theta_{\bc,i}^\dagger$, we see that the asymptotic decay rate in direction $\bc$ is given depending on the geometrical relation between $\mathrm{Q}_1$ and $\mathrm{Q}_2$, as follows  (see Figs.\  \ref{fig:fig34} and \ref{fig:classification}). 
%
\begin{figure}[tb]
\begin{center}
\includegraphics[width=160mm,trim=0 0 0 0]{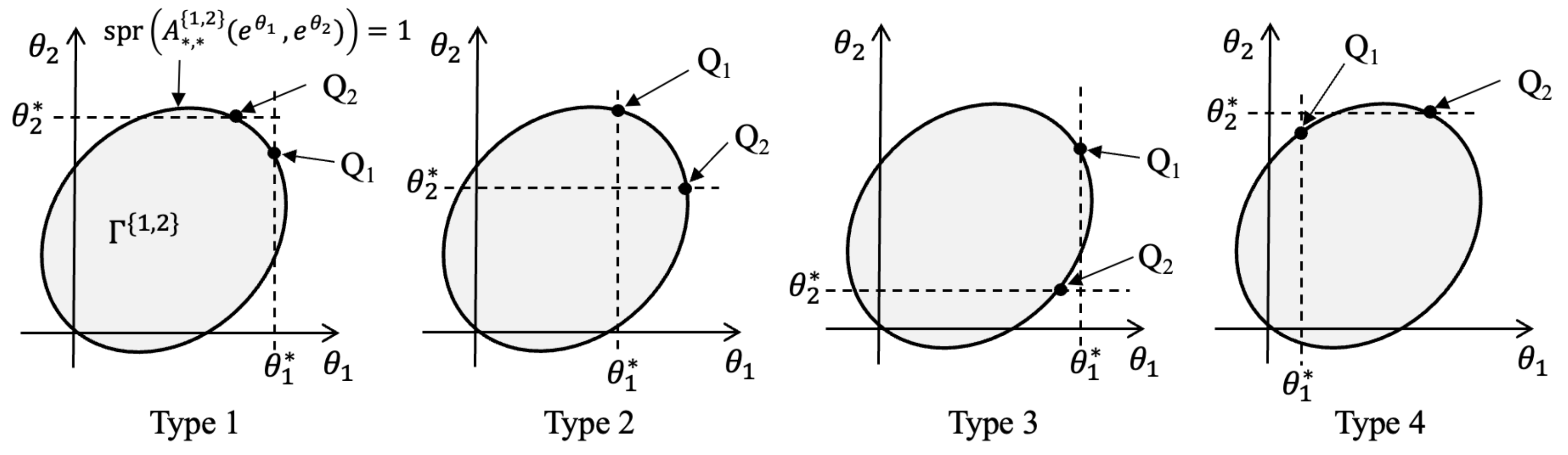} 
\caption{Classification}
\label{fig:classification}
\end{center}
\end{figure}
\begin{itemize}
\item \textit{Type 1}. 
If $-c_1/c_2\le \bar{\eta}'_2(\theta_1^*)$, then $\xi_{\bc}=c_1 \theta_1^* + c_2 \bar{\eta}_2(\theta_1^*) $, where $\bar{\eta}'_2(x) = (d/dx)\bar{\eta}_2(x)$;  
If $-c_2/c_1\le\bar{\eta}'_1(\theta_2^*)$, then $\xi_{\bc}=c_1 \bar{\eta}_1(\theta_2^*) +c_2 \theta_2^*$, where $\bar{\eta}'_1(x) = (d/dx)\bar{\eta}_1(x)$; 
Otherwise (i.e., $\bar{\eta}'_2(\theta_1^*) < -c_1/c_2< 1/\bar{\eta}'_1(\theta_2^*)$), $\xi_{\bc} = \theta_{\bc}^{max}$.
\item \textit{Type 2}. 
If $-c_1/c_2\le (\theta_2^*-\bar{\eta}_2(\theta_1^*))/(\bar{\eta}_1(\theta_2^*)-\theta_1^*)$, then  $\xi_{\bc}=c_1 \theta_1^* + c_2 \bar{\eta}_2(\theta_1^*) $; 
Otherwise (i.e., $-c_1/c_2> (\theta_2^*-\bar{\eta}_2(\theta_1^*))/(\bar{\eta}_1(\theta_2^*)-\theta_1^*))$, $\xi_{\bc}=c_1 \bar{\eta}_1(\theta_2^*) +c_2 \theta_2^*$. 
\item \textit{Type 3}. 
$\xi_{\bc}=c_1 \bar{\eta}_1(\theta_2^*)+c_2 \theta_2^*$.
\item \textit{Type 4}. 
$\xi_{\bc}=c_1 \theta_1^*+c_2 \bar{\eta}_2(\theta_1^*)$. 
\end{itemize}

This also holds for the case where $\bc=(1,0)$ or $\bc=(0,1)$.

%
%
\subsection{An example} \label{sec:example}

We consider the same queueing model as that used in Ozawa and Kobayashi \cite{Ozawa18b}. 
It is a single-server two-queue model in which the server visits the queues alternatively, serves one queue (queue 1) according to a 1-limited service and the other queue (queue 2) according to an exhaustive-type $K$-limited service (see Fig.\ \ref{fig:two_queue}). 
Customers arrive at queue 1 (resp.\ queue 2) according to a Poisson process with intensity $\lambda_1$ (resp.\ $\lambda_2$).  Service times are exponentially distributed with mean $1/\mu_1$ in queue 1 ($1/\mu_2$ in queue 2). The arrival processes and service times are mutually independent. 
We refer to this model as a $(1,K)$-limited service model. 
In this model, the asymptotic decay rate $\xi_{\bc}$ indicates how the joint queue length probability in steady state decreases as the queue lengths of queue 1 and queue 2 simultaneously enlarge.

%
\begin{figure}[tb]
\begin{center}
\includegraphics[width=80mm,trim=0 0 0 0]{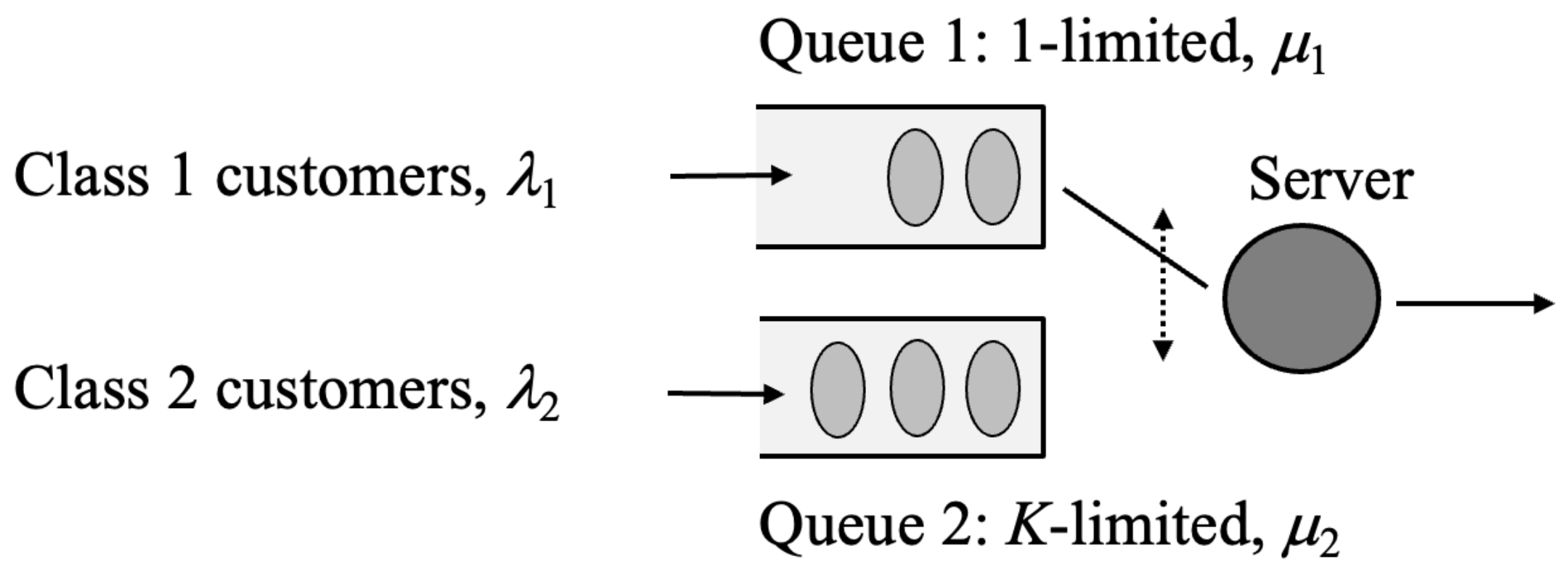} 
\caption{Single server two-queue model}
\label{fig:two_queue}
\end{center}
\end{figure}
%
\begin{figure}[tb]
\begin{center}
\includegraphics[width=160mm,trim=0 0 0 0]{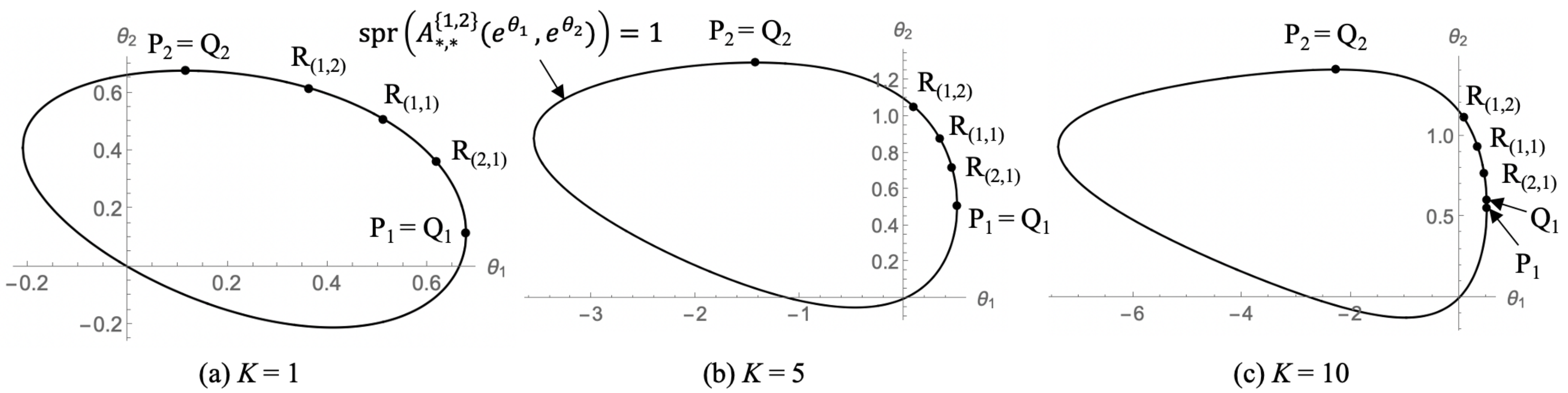} 
\\
{\small 
$\lambda_1=\lambda_2=0.3$, $\mu_1=\mu_2=1$, 
$\mathrm{P}_1=(\theta_1^{max},\bar{\eta}_2(\theta_1^{max}))$, $\mathrm{Q}_1=(\theta_1^*,\bar{\eta}_2(\theta_1^*))$, 
$\mathrm{P}_2=(\bar{\eta}_1(\theta_2^{max}),\theta_2^{max})$, $\mathrm{Q}_2=(\bar{\eta}_1(\theta_2^*),\theta_2^*)$, 
$\mathrm{R}_{\bc}=(\eta^R_{\bc,1}(\theta_{\bc}^{max}),\eta^R_{\bc,2}(\theta_{\bc}^{max}))$
}
\caption{Points on the closed curve $\cp(A^{\{1,2\}}_{*,*}(e^{\theta_1},e^{\theta_2}))=1$}
\label{fig:example1}
\end{center}
\end{figure}
%
\begin{figure}[tb]
\begin{center}
\includegraphics[width=160mm,trim=0 0 0 0]{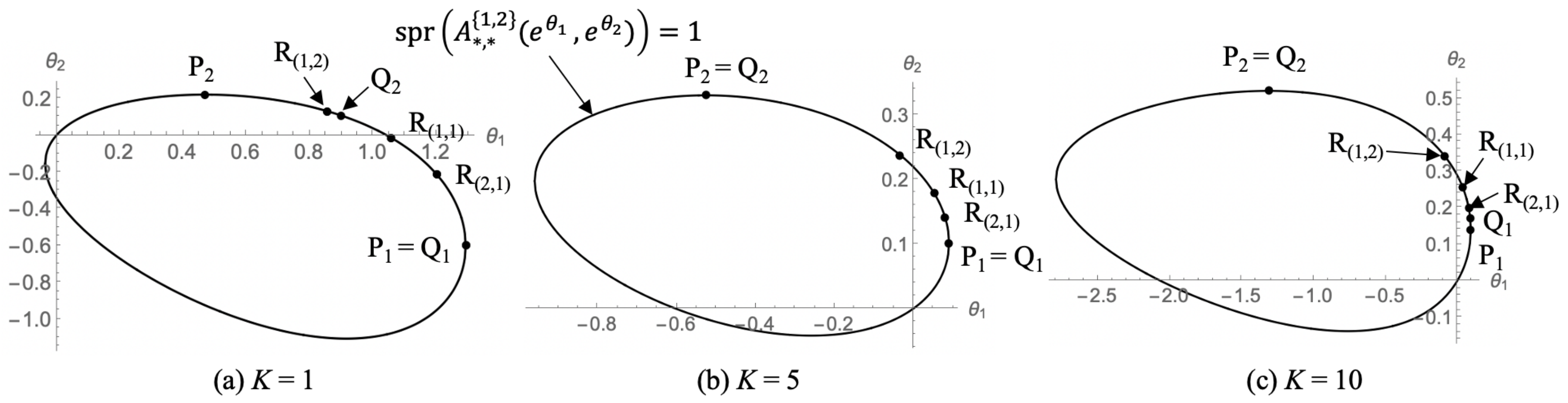} 
\\
{\small 
$\lambda_1=0.24,\ \lambda_2=0.7$, $\mu_1=1.2,\ \mu_2=1$, 
$\mathrm{P}_1=(\theta_1^{max},\bar{\eta}_2(\theta_1^{max}))$, $\mathrm{Q}_1=(\theta_1^*,\bar{\eta}_2(\theta_1^*))$, 
$\mathrm{P}_2=(\bar{\eta}_1(\theta_2^{max}),\theta_2^{max})$, $\mathrm{Q}_2=(\bar{\eta}_1(\theta_2^*),\theta_2^*)$, 
$\mathrm{R}_{\bc}=(\eta^R_{\bc,1}(\theta_{\bc}^{max}),\eta^R_{\bc,2}(\theta_{\bc}^{max}))$
}
\caption{Points on the closed curve $\cp(A^{\{1,2\}}_{*,*}(e^{\theta_1},e^{\theta_2}))=1$}
\label{fig:example2}
\end{center}
\end{figure}

Let $X_1(t)$ be the number of customers in queue 1 at time $t$, $X_2(t)$ that of customers in queue 2 and $J(t) \in S_0=\{0, 1, ..., K\}$ the server state. 
When $X_1(t)=X_2(t)=0$, $J(t)$ takes one of the states in $S_0$ at random; 
When $X_1(t)\ge 1$ and $X_2(t)=0$, it also takes one of the states in $S_0$ at random; 
When $X_1(t)=0$ and $X_2(t)\ge 1$, it takes the state of $0$ or $1$ at random if the server is serving the $K$-th customer in queue 2 during a visit of the server at queue 2 and takes the state of $j\in\{2, ..., K\}$ if the server is serving the $(K-j+1)$-th customer in queue 2; 
When $X_1(t)\ge 1$ and $X_2(t)\ge 1$, it takes the state of $0$ if the server is serving a customer in queue 1 and takes the state of $j\in\{1, ..., K\}$ if the server is serving the $(K-j+1)$-th customer in queue 2 during a visit of the server at queue 2. 
The process $\{(X_1(t), X_2(t), J(t))\}$ becomes a continuous-time 2d-QBD process on the state space $\mathbb{Z}_+^2\times S_0$. By uniformization with parameter $\nu=\lambda+\mu_1+\mu_2$, we obtain the corresponding discrete-time 2d-QBD process, $\{(X_{1,n}, X_{2,n}, J_n)\}$. For the description of the transition probability blocks such as $A^{\{1,2\}}_{i,j}$, see Ref.\ \cite{Ozawa18b}.
This $(1,K)$-limited service model satisfies Assumptions \ref{as:QBD_irreducible}, \ref{as:MAprocess_irreducible}, \ref{as:Y12_onZpZp_irreducible} and \ref{as:Y12_onZpmZmp_irreducible}.

In numerical experiments, we treat two cases: a symmetric parameter case (see Fig.\ \ref{fig:example1} and Table \ref{tab:example1}) and an asymmetric parameter case (see Fig.\ \ref{fig:example2} and Table \ref{tab:example2}). In both the cases, the value of $K$ is set at $1$, $5$ or $10$. 
In Figs.\ \ref{fig:example1} and \ref{fig:example2}, the closed curves of $\spr(A^{\{1,2\}}_{*,*}(e^{\theta_1},e^{\theta_2}))=1$ are drawn with points $\mathrm{Q}_1$ and $\mathrm{Q}_2$. Define points $\mathrm{P}_1$, $\mathrm{P}_2$ and $\mathrm{R}_{\bc}$ as $\mathrm{P}_1=(\theta_1^{max},\bar{\eta}_2(\theta_1^{max}))$,  $\mathrm{P}_2=(\bar{\eta}_1(\theta_2^{max}),\theta_2^{max})$ and $\mathrm{R}_{\bc}=(\eta^R_{\bc,1}(\theta_{\bc}^{max}),\eta^R_{\bc,2}(\theta_{\bc}^{max}))$, respectively. For the definition of $\theta_1^{max}$ and $\theta_2^{max}$, see Appendix \ref{sec:block_2dQBD_results}. These points are also written on the figures. 
From the figures, we see that all the cases are classified into Type 1. 
If $\mathrm{Q}_1=\mathrm{P}_1$ and $\mathrm{Q}_2=\mathrm{P}_2$ (see Fig.\ \ref{fig:example1} (a), (b) and Fig.\ \ref{fig:example2} (b)), then, for any $\bc=(c_1,c_2)\in\mathbb{N}^2$, $\xi_{\bc}$ is given by $\theta_{\bc}^{max}$. 
On the other hand, in the  symmetric case of $K=10$ (see Fig.\ \ref{fig:example1} (c)), $\xi_{\bc}$ is given by $\theta_{\bc}^{max}$ only if $-c_1/c_2> \bar{\eta}'_2(\theta_1^*)=-9.87$; In the asymmetric case of $K=1$ (see Fig.\ \ref{fig:example2} (a)), it is given by $\theta_{\bc}^{max}$ only if $-c_2/c_1> \bar{\eta}'_1(\theta_2^*)=-1.73$; In that of $K=10$ (see Fig.\ \ref{fig:example2} (c)), it is given by $\theta_{\bc}^{max}$ only if $-c_1/c_2> \bar{\eta}'_2(\theta_1^*)=-3.88$.
Tables \ref{tab:example1} and \ref{tab:example2} shows the normalized values of $\xi_{\bc}$, i.e., $\xi_{\bc}/\parallel\!\bc\!\parallel$, where $\parallel\!\bc\!\parallel=\sqrt{c_1^2+c_2^2}$. From the tables, it can be seen how the values of the asymptotic decay rate vary according to the direction vector. 
%
%
\begin{table}[htbp]
\caption{Asymptotic decay rates ($\lambda_1=\lambda_2=0.3$, $\mu_1=\mu_2=1$)}
\begin{center}
\begin{tabular}{c|cc|cc|ccccc}
$K$ & $\theta_1^{max}$ & $\theta_1^*$ & $\theta_2^{max}$ & $\theta_2^*$ & $\xi_{(1,0)}$ & $\xi_{(2,1)}/\sqrt{5}$ & $\xi_{(1,1)}/\sqrt{2}$ & $\xi_{(1,2)}/\sqrt{5}$ & $\xi_{(0,1)}$  \cr \hline
$1$ & $0.677$ & $\leftarrow$ & $0.677$ & $\leftarrow$ & $0.667$  & $0.714$ & $0.722$ & $0.714$ & $0.677$ \cr
$5$ & $0.511$ & $\leftarrow$ & $1.30$ & $\leftarrow$ & $0.511$  & $0.734$ & $0.866$ & $0.986$ & $1.30$ \cr
$10$ & $0.513$ & $0.511$ & $1.41$ & $\leftarrow$ & $0.511$  & $0.757$ & $0.901$ & $1.03$ & $1.41$
\end{tabular}
\end{center}
\label{tab:example1}
\end{table}%
%
\begin{table}[htbp]
\caption{Asymptotic decay rates ($\lambda_1=0.24$, $\lambda_2=0.7$, $\mu_1=1.2$, $\mu_2=1$)}
\begin{center}
\begin{tabular}{c|cc|cc|ccccc}
$K$ & $\theta_1^{max}$ & $\theta_1^*$ & $\theta_2^{max}$ & $\theta_2^*$ & $\xi_{(1,0)}$ & $\xi_{(2,1)}/\sqrt{5}$ & $\xi_{(1,1)}/\sqrt{2}$ & $\xi_{(1,2)}/\sqrt{5}$ & $\xi_{(0,1)}$  \cr \hline
$1$ & $1.29$ & $\leftarrow$ & $0.223$ & $0.110$ & $1.29$  & $0.98$ & $0.740$ & $0.500$ & $0.110$ \cr
$5$ & $0.091$ & $\leftarrow$ & $0.331$ & $\leftarrow$ & $0.091$  & $0.136$ & $0.164$ & $0.198$ & $0.331$ \cr
$10$ & $0.094$ & $0.090$ & $0.520$ & $\leftarrow$ & $0.090$  & $0.161$ & $0.208$ & $0.267$ & $0.520$
\end{tabular}
\end{center}
\label{tab:example2}
\end{table}%

%
%
\section{Concluding remark} \label{sec:conclusion}

The large deviation techniques are often used for investigating asymptotics of the stationary distributions in Markov processes on the positive quadrant (see, for example, Miyazawa \cite{Miyazawa11} and references therein). In analysis using them, the upper and lower bounds for the asymptotic decay rates are represented in terms of the large deviation rate function, and that rate function is given by the variational problem minimizing the total variances of the critical path. Set $\mathbb{D}^{\{1\}}=\{(x_1,x_2)\in\mathbb{R}^2; x_1>0, x_2=0 \}$, $\mathbb{D}^{\{2\}}=\{(x_1,x_2)\in\mathbb{R}^2; x_1=0,\,x_2>0 \}$ and $\mathbb{D}^{\{1,2\}}=\{(x_1,x_2)\in\mathbb{R}^2; x_1>0,\, x_2>0\}$. The following three kinds of path of point $\bp$ moving from the origin to a positive point $\bp_0\in\mathbb{D}^{\{1,2\}}$ are often used as options for the critical path (see, for example, Dai and Miyazawa \cite{Dai13} in the case of SRBM).
\begin{itemize}
\item Type-0 path: $\bp$ directly moves from the origin to $\bp_0$ through $\mathbb{D}^{\{1,2\}}$. 
\item Type-1 path: First, $\bp$ moves from the origin to some point on $\mathbb{D}^{\{1\}}$ through $\mathbb{D}^{\{1\}}$ and then it moves to $\bp_0$ through $\mathbb{D}^{\{1,2\}}$. 
\item Type-2 path: Replace $\mathbb{D}^{\{1\}}$ with $\mathbb{D}^{\{2\}}$ in the definition of Type-1 path. 
\end{itemize}
In our analysis, the generating function $\bvarphi^{\bc}(z)$ has been divided into three parts: $\bvarphi^{\bc}_0(z)$, $\bvarphi^{\bc}_1(z)$ and $\bvarphi^{\bc}_2(z)$, through compensation equation \eqref{eq:tildenux}.  
In some sense, $\bvarphi^{\bc}_0(z)$ evaluates Type-0 paths and ``$\xi_{\bc}=\theta_{\bc}^{max}$" corresponds to the case where the critical path is of Type-0; $\bvarphi^{\bc}_1(z)$ evaluates Type-1 paths and ``$\xi_{\bc}=\theta_{\bc,1}^\dagger<\theta_{\bc}^{max}$" corresponds to the case where the critical path is of Type-1; $\bvarphi^{\bc}_2(z)$ evaluates Type-2 paths and ``$\xi_{\bc}=\theta_{\bc,2}^\dagger<\theta_{\bc}^{max}$" corresponds to the case where the critical path is of Type-2. This analogy gives us some insight to investigate asymptotics of the stationary tail distributions in higher-dimensional QBD processes. 
For related topics with respect to queueing networks, see Foley and McDonald \cite{Foley05}, where paths called jitter, bridge and cascade ones are considered. Type-0 paths above correspond to jitter ones, and Type-1 and Type-2 paths to cascade ones.

%
%

%
%
\appendix

%
\section{Asymptotic properties of the block state process} \label{sec:block_2dQBD_results}

For $\bb=(b_1,b_2)\in\mathbb{N}^2$, let $\{{}^{\bb}\bY_n\}=\{({}^{\bb}\!\bX_n,({}^{\bb}\!\bM_n,{}^{\bb}\!J_n))\}$ be the $\bb$-block state process derived from a 2d-QBD process $\{\bY_n\}=\{(\bX_n,J_n)\}$, introduced in Subsection \ref{sec:block_state_process}. 
Since the $\bb$-block state process is also a 2d-QBD process, we obtain by the results of Refs. \cite{Ozawa13,Ozawa18,Ozawa21} the following.

Define vector generating functions ${}^{\bb}\bnu_{(*,0)}(z)$ and ${}^{\bb}\bnu_{(0,*)}(z)$ as
\[
{}^{\bb}\bnu_{(*,0)}(z) = \sum_{k=1}^\infty z^k\, {}^{\bb}\bnu_{(k,0)},\quad 
{}^{\bb}\bnu_{(0,*)}(z) = \sum_{k=1}^\infty z^k\, {}^{\bb}\bnu_{(0,k)}. 
\]
%
\begin{figure}[t]
\begin{center}
\includegraphics[width=80mm,trim=0 0 0 0]{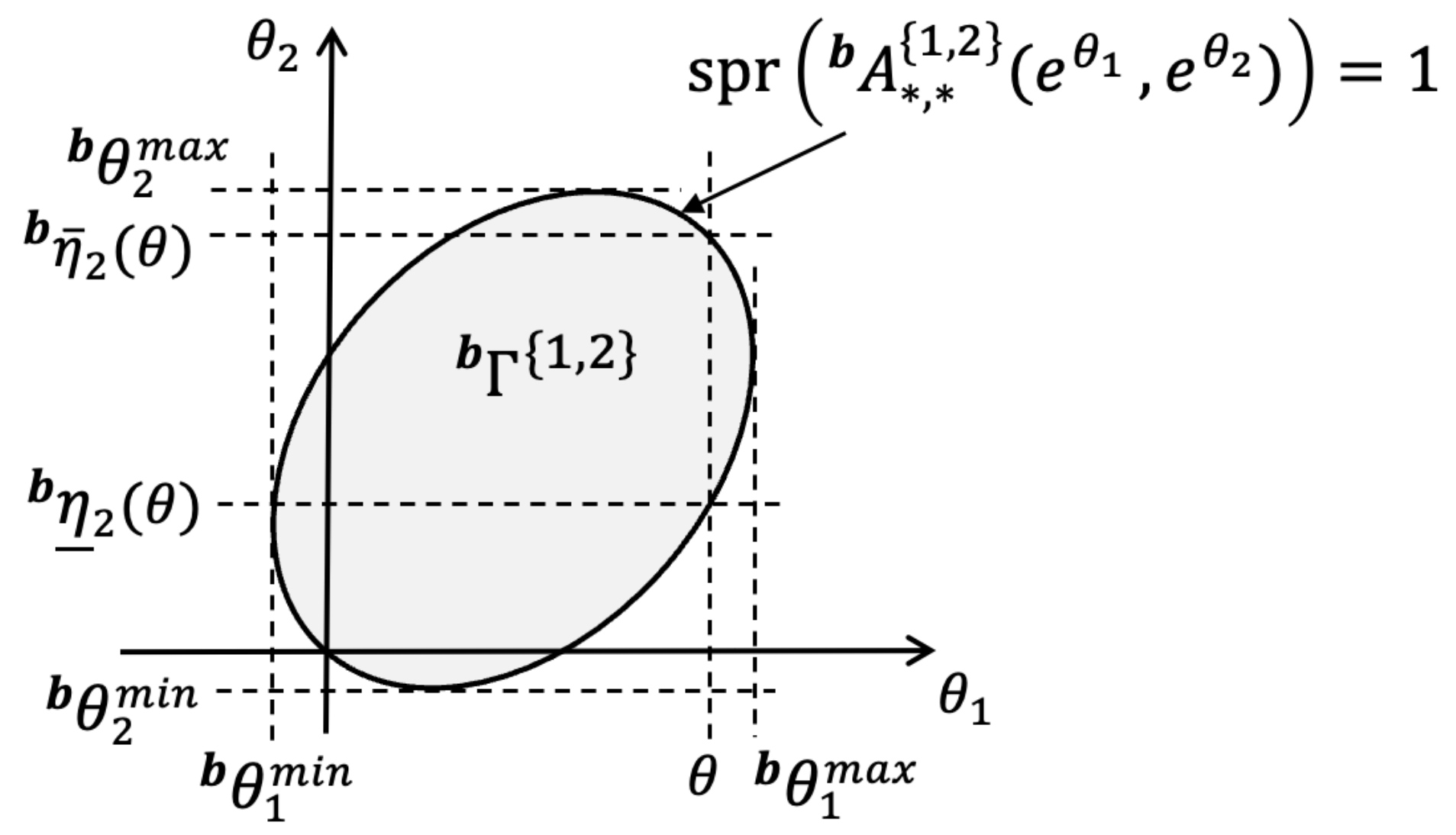} 
\caption{Domain ${}^{\bb}\Gamma^{\{1,2\}}$}
\label{fig:figA1}
\end{center}
\end{figure}
%
Define a matrix function ${}^{\bb}\!A^{\{1,2\}}_{*,*}(z_1,z_2)$ as
 \[
 {}^{\bb}\!A^{\{1,2\}}_{*,*}(z_1,z_2) = \sum_{i_1,i_2\in\{-1,0,1\}} z_1^{i_1} z_2^{i_2}\, {}^{\bb}\!A^{\{1,2\}}_{i_1,i_2}, 
 \]
and a domain ${}^{\bb}\Gamma^{\{1,2\}}$ as 
\[
{}^{\bb}{\Gamma}^{\{1,2\}} = \{(\theta_1,\theta_2)\in\mathbb{R}^2; \spr({}^{\bb}\!A^{\{1,2\}}_{*,*}(e^{\theta_1},e^{\theta_2}))<1\}. 
\]
By Lemma A.1 of Ozawa \cite{Ozawa21}, $\spr({}^{\bb}\!A^{\{1,2\}}_{*,*}(e^{\theta_1},e^{\theta_2}))$ is log-convex in $(\theta_1,\theta_2)$, and the closure of ${}^{\bb}{\Gamma}^{\{1,2\}}$ is a convex set. 
Define the extreme values of ${}^{\bb}{\Gamma}^{\{1,2\}}$, ${}^{\bb}\theta^{min}_i$  and ${}^{\bb}\theta^{max}_i$ for $i\in\{1,2\}$, as
\begin{equation}
{}^{\bb}\theta^{min}_i = \inf\{\theta_i; (\theta_1,\theta_2)\in{}^{\bb}{\Gamma}^{\{1,2\}} \},\quad 
{}^{\bb}\theta^{max}_i = \sup\{\theta_i; (\theta_1,\theta_2)\in{}^{\bb}{\Gamma}^{\{1,2\}} \}. 
\label{eq:thetai_maxmin}
\end{equation}
For $\theta_1\in[{}^{\bb}\theta^{min}_1,{}^{\bb}\theta^{max}_1]$, let ${}^{\bb}\underline{\eta}_2(\theta_1)$ and ${}^{\bb}\bar{\eta}_2(\theta_1)$ be the real two roots to equation 
\begin{equation}
\spr({}^{\bb}\!A^{\{1,2\}}_{*,*}(e^{\theta_1},e^{\theta_2}))=1,
\label{eq:sprA12_eq1}
\end{equation}
counting multiplicity, where ${}^{\bb}\underline{\eta}_2(\theta_1)\le {}^{\bb}\bar{\eta}_2(\theta_1)$ (see Fig.\,\ref{fig:figA1}). For $\theta_2\in[{}^{\bb}\theta^{min}_2,{}^{\bb}\theta^{max}_2]$, ${}^{\bb}\underline{\eta}_1(\theta_2)$ and ${}^{\bb}\bar{\eta}_1(\theta_2)$ are analogously defined. 
Hereafter, if $\bb=(1,1)$, we omit the left superscript $\bb$; for example, ${}^{\bb}\!A^{\{1,2\}}_{*,*}(e^{\theta_1},e^{\theta_2})$ is denoted by $A^{\{1,2\}}_{*,*}(e^{\theta_1},e^{\theta_2})$ and ${}^{\bb}{\Gamma}^{\{1,2\}}$ by ${\Gamma}^{\{1,2\}}$ ($A^{\{1,2\}}_{*,*}(e^{\theta_1},e^{\theta_2})$ and ${\Gamma}^{\{1,2\}}$ have already been defined in Section \ref{sec:intro}). By Proposition 4.2 of Ozawa \cite{Ozawa21}, we have
\begin{equation}
\spr(A^{\{1,2\}}_{*,*}(e^{\theta_1},e^{\theta_2})) = \spr({}^{\bb}\!A^{\{1,2\}}_{*,*}(e^{b_1 \theta_1},e^{b_2 \theta_2})). 
\label{eq:sprA_sprbA_relation}
\end{equation}
This implies that, for example, ${}^{\bb}\theta^{max}_1=b_1 \theta^{max}_1$ and ${}^{\bb}\underline{\eta}_2(b_1\theta_1)=b_2 \underline{\eta}_2(\theta_1)$.

%
For $i\in\{-1,0,1\}$ and $i'\in\{0,1\}$, define matrix functions ${}^{\bb}\!A_{i',*}^\emptyset(z)$, ${}^{\bb}\!A_{i,*}^{\{1\}}(z)$, ${}^{\bb}\!A_{i',*}^{\{2\}}(z)$ and ${}^{\bb}\!A_{i,*}^{\{1,2\}}(z)$ as
\begin{align*}
&{}^{\bb}\!A_{i',*}^{\emptyset}(z) = {}^{\bb}\!A_{i',0}^\emptyset+z\, {}^{\bb}\!A_{i',1}^\emptyset,\quad 
{}^{\bb}\!A_{i,*}^{\{1\}}(z) = {}^{\bb}\!A_{i,0}^{\{1\}}+z\, {}^{\bb}\!A_{i,1}^{\{1\}}, \\ 
&{}^{\bb}\!A_{i',*}^{\{2\}}(z) = z^{-1}\, {}^{\bb}\!A_{i',-1}^{\{2\}}+ {}^{\bb}\!A_{i',0}^{\{2\}}+z\, {}^{\bb}\!A_{i',1}^{\{2\}},\quad 
{}^{\bb}\!A_{i,*}^{\{1,2\}}(z) = z^{-1}\, {}^{\bb}\!A_{i,-1}^{\{1,2\}}+ {}^{\bb}\!A_{i,0}^{\{1,2\}}+z\, {}^{\bb}\!A_{i,1}^{\{1,2\}}.
 \end{align*}
 For $i\in\{-1,0,1\}$ and $i'\in\{0,1\}$, analogously define matrix functions ${}^{\bb}\!A_{*,i'}^\emptyset(z)$, ${}^{\bb}\!A_{*,i'}^{\{1\}}(z)$, ${}^{\bb}\!A_{*,i}^{\{2\}}(z)$ and ${}^{\bb}\!A_{*,i}^{\{1,2\}}(z)$. 
For $z_1\in[e^{{}^{\bb}\theta^{min}_1},e^{{}^{\bb}\theta^{max}_1}]$ and $z_2\in[e^{{}^{\bb}\theta^{min}_2},e^{{}^{\bb}\theta^{max}_2}]$, let ${}^{\bb}G_1(z_1)$ and ${}^{\bb}G_2(z_2)$ be the minimum nonnegative solutions to quadratic matrix equations \eqref{eq:bG1theta} and \eqref{eq:bG2theta}, respectively:
\begin{align}
&{}^{\bb}\!A_{*,-1}^{\{1,2\}}(z_1) + {}^{\bb}\!A_{*,0}^{\{1,2\}}(z_1) X + {}^{\bb}\!A_{*,1}^{\{1,2\}}(z_1) X^2 = X, \label{eq:bG1theta} \\
&{}^{\bb}\!A_{-1,*}^{\{1,2\}}(z_2) + {}^{\bb}\!A_{0,*}^{\{1,2\}}(z_2) X + {}^{\bb}\!A_{1,*}^{\{1,2\}}(z_2) X^2 = X, \label{eq:bG2theta}
\end{align}
where ${}^{\bb}G_1(z_1)$ and ${}^{\bb}G_2(z_2)$ are called G-matrices in the queueing theory. By Lemma 2.5 of Ozawa \cite{Ozawa21}, we have
\begin{equation}
\spr({}^{\bb}G_1(e^{\theta_1})) = e^{{}^{\bb}\underline{\eta}_2(\theta_1)},\quad 
\spr({}^{\bb}G_2(e^{\theta_2})) = e^{{}^{\bb}\underline{\eta}_1(\theta_2)}. 
\end{equation}
%
Define matrix functions ${}^{\bb}U_1(z_1)$ and ${}^{\bb}U_2(z_2)$ as
\[
{}^{\bb}U_1(z_1) = {}^{\bb}\!A_{*,0}^{\{1\}}(z_1) + {}^{\bb}\!A_{*,1}^{\{1\}}(z_1)\, {}^{\bb}G_1(z_1),\quad 
{}^{\bb}U_2(z_2) = {}^{\bb}\!A_{0,*}^{\{2\}}(z_2) + {}^{\bb}\!A_{1,*}^{\{2\}}(z_2)\, {}^{\bb}G_2(z_2),  
\]
and, for $i\in\{1,2\}$, a real value ${}^{\bb}\theta^*_i$ as
\[
{}^{\bb}\theta^*_i = \sup\{ \theta\in[{}^{\bb}\theta^{min}_i,{}^{\bb}\theta^{max}_i]; \spr({}^{\bb}U_i(e^{\theta}))< 1\}.
\]
Define real values ${}^{\bb}\theta^\dagger_1$ and  ${}^{\bb}\theta^\dagger_2$ as
\[
{}^{\bb}\theta^\dagger_1 =  \max\{ \theta\in[{}^{\bb}\theta^{min}_1,{}^{\bb}\theta^{max}_1]; {}^{\bb}\underline{\eta}_2(\theta) \le {}^{\bb}\theta^*_2 \},\quad 
{}^{\bb}\theta^\dagger_2 =  \max\{ \theta\in[{}^{\bb}\theta^{min}_2,{}^{\bb}\theta^{max}_2]; {}^{\bb}\underline{\eta}_1(\theta) \le {}^{\bb}\theta^*_1 \}.
\]
Note that if $\bb=(1,1)$, then, for $i\in\{1,2\}$, inequality $\spr({}^{\bb}U_i(e^\theta))=\spr(U_i(e^\theta))< 1$ is equivalent to $\cp(\bar{A}_*^{\{i\}}(e^\theta))>1$ (for the definition of $\bar{A}_*^{\{i\}}(z)$, see Section \ref{sec:intro}). Hence, for $i\in\{1,2\}$, $\Gamma^{\{i\}}$ defined in Section \ref{sec:intro} satisfies 
\begin{equation}
\Gamma^{\{i\}}=\{(\theta_1,\theta_2)\in\mathbb{R}^2; \spr(U_i(e^{\theta_i}))< 1\}.
\end{equation}
Furthermore, for $i\in\{1,2\}$, 
\begin{equation}
\theta_i^*=\sup\{\theta_i; (\theta_1,\theta_2)\in\Gamma^{\{i\}} \},\quad 
\theta_i^\dagger=\sup\{\theta_i; (\theta_1,\theta_2)\in\Gamma^{\{3-i\}}\cap\Gamma^{\{1,2\}} \}. 
\end{equation}
By Lemma 2.6 of Ozawa and Kobayashi \cite{Ozawa18}, we have the following.
\begin{lemma}
The asymptotic decay rates ${}^{\bb}\xi_{(1,0)}$ and ${}^{\bb}\xi_{(0,1)}$ are given by 
\begin{equation}
{}^{\bb}\xi_{(1,0)} = \min\{ {}^{\bb}\theta^*_1,\,{}^{\bb}\theta^\dagger_1 \},\quad 
{}^{\bb}\xi_{(0,1)} = \min\{ {}^{\bb}\theta^*_2,\, {}^{\bb}\theta^\dagger_2\}.  
\end{equation}
\end{lemma}

By \eqref{eq:sprA_sprbA_relation}, we have 
\begin{align}
&{}^{\bb}\theta^*_1=b_1 \theta^*_1,\quad
{}^{\bb}\theta^\dagger_1=b_1 \theta^\dagger_1,\quad 
{}^{\bb}\xi_{(1,0)} = b_1 \xi_{(1,0)}, \\
&{}^{\bb}\theta^*_2=b_2 \theta^*_2,\quad
{}^{\bb}\theta^\dagger_2=b_2 \theta^\dagger_2,\quad 
{}^{\bb}\xi_{(0,1)} = b_2 \xi_{(0,1)}.
\end{align}
%
By Proposition 3.3 of Ozawa and Kobayashi \cite{Ozawa18}, ${}^{\bb}\bnu_{(0,*)}(z)$ satisfies the following equation:
\begin{align}
{}^{\bb}\bnu_{(0,*)}(z) 
&= \sum_{k=1}^\infty {}^{\bb}\bnu_{(k,0)} \sum_{i\in\{-1,0,1\}} ({}^{\bb}\!A^{\{1\}}_{i,*}(z)-{}^{\bb}\!A_{i,*}^{\{1,2\}}(z))\, {}^{\bb}G_2(z)^{k+i} \left(I-{}^{\bb}U_2(z) \right)^{-1}  \cr
&\qquad + {}^{\bb}\bnu_{(0,0)}  \sum_{i\in\{0,1\}} ({}^{\bb}\!A^{\emptyset}_{i,*}(z)-{}^{\bb}\!A_{i,*}^{\{2\}}(z))\, {}^{\bb}G_2(z)^i \left(I-{}^{\bb}U_2(z)\right)^{-1}. 
\label{eq:bnu0s_eq}
\end{align}
This is a kind of compensation equation. An equation similar to \eqref{eq:bnu0s_eq} also holds for ${}^{\bb}\bnu_{(*,0)}(z) $.

%
%

%
\section{Proof of Proposition \ref{pr:xi_equality}} \label{sec:xi_equality}

\begin{proof}[Proof of Proposition \ref{pr:xi_equality}]
For a sequence $\{a_n\}_{n\ge 1}$, we denote by $\bar{a}_k$ the partial sum of the sequence defined as $\bar{a}_k=\sum_{n=1}^k a_n$. 
Let $\bc=(c_1,c_2)$ be a vector of positive integers. Let $(\bx,j)$ and $(\bx',j')$ be arbitrary states in $\mathbb{N}^2\times S_0$ such that $(\bx,j)\ne (\bx',j')$. 
Since the induced MA-process $\{\bY^{\{1,2\}}_n\}$ is irreducible, there exist a $k_0\ge 1$, $n_0\ge 1$ and sequence $\{(l_n,m_n,j_n)\in\{-1,0,1\}^2\times S_0; 1\le n\le n_0\}$ such that $\bx+k_0\bc+(\bar{l}_k,\bar{l}_k)>(0,0)$ and $(\bx+k_0\bc+(\bar{l}_k,\bar{l}_k),j_k)\ne(\bx'+k_0\bc,j')$ for every integer $k\in[1,n_0-1]$, $(\bx+k_0\bc+(\bar{l}_{n_0},\bar{m}_{n_0}),j_{n_0})=(\bx'+k_0\bc,j')$ and  
\[
p^*= [A^{\{1,2\}}_{l_1,m_1}]_{j,j_1} \prod_{n=2}^{n_0-1} [A^{\{1,2\}}_{l_n,m_n}]_{j_{n-1},j_n} [A^{\{1,2\}}_{l_{n_0},m_{n_0}}]_{j_{n_0-1},j'}>0.   
\] 
Such a sequence gives a path from $\bY^{\{1,2\}}_0=(\bx+k_0\bc,j)$ to $\bY^{\{1,2\}}_{n_0}=(\bx'+k_0\bc,j')$ on $\mathbb{N}^2\times S_0$, and that path is also a path from $\bY_0=(\bx+k_0\bc,j)$ to $\bY_{n_0}=(\bx'+k_0\bc,j')$ in the original 2d-QBD process $\{\bY_n\}$.
For $k\ge 1$, let $\tau^{(k)}$ be the first hitting time to the state $(\bx+k\bc,j)$ in $\{\bY_n\}$, i.e., $\tau^{(k)}=\inf\{n\ge 1; \bY_n=(\bx+k\bc,j)\}$, and denote by $(q^{(k)}_{(\bx'',j'')}; (\bx'',j'')\in\mathbb{Z}_+^2\times S_0)$ the occupation measure defined as
\[
q^{(k)}_{(\bx'',j'')} = \mathbb{E}\Big( \sum_{n=0}^{\tau^{(k)}-1} 1(\bY_n=(\bx'',j'')) \,|\, \bY_0=(\bx+k\bc,j) \Big).
\]
Then, we have 
\begin{equation}
\nu_{(\bx'+k\bc,j')} = q^{(k)}_{(\bx'+k\bc,j')} \nu_{(\bx+k\bc,j)}.
\end{equation}
Due to the space homogeneity of $\{{\bY}^{\{1,2\}}_n\}$ with respect to the additive part, for every $k\ge k_0$, there exists a path from $\bY^{\{1,2\}}_0=(\bx+k\bc,j)$ to $\bY^{\{1,2\}}_{n_0}=(\bx'+k\bc,j')$ given by the same sequence as $\{(l_n,m_n,j_n)\in\{-1,0,1\}^2\times S_0; 1\le n\le n_0\}$ mentioned above, and it is also a path from $\bY_0=(\bx+k\bc,j)$ to $\bY_{n_0}=(\bx'+k\bc,j')$ in the original 2d-QBD process. Hence, we have $q^{(k)}_{(\bx'+k\bc,j')}\ge p^*$ and obtain
\begin{equation}
\nu_{(\bx'+k\bc,j')}\ge p^* \nu_{(\bx+k\bc,j)}, 
\end{equation}
where $p^*$ does not depend on $k$. This leads us to $\underline{\xi}_{\bc}(\bx',j')\le \underline{\xi}_{\bc}(\bx,j)$ and  $\bar\xi_{\bc}(\bx',j')\le \bar\xi_{\bc}(\bx,j)$. Interchanging $(\bx,j)$ with $(\bx',j')$, we analogously obtain $\underline{\xi}_{\bc}(\bx,j)\le \underline{\xi}_{\bc}(\bx',j')$ and $\bar\xi_{\bc}(\bx,j)\le \bar\xi_{\bc}(\bx',j')$. This completes the proof. 
\end{proof}


\begin{thebibliography}{99}
%
\bibitem{Bini05}
Bini, D.A., Latouche, G.\ and Meini, B.,
\textit{Numerical Solution of Structured Markov Chains}, 
Oxford University Press, Oxford (2005).
%
\bibitem{Borovkov01}
Borovkov, A.A.\ and Mogul'ski\u\i, A.A.,
Large deviations for Markov chains in the positive quadrant, 
\textit{Russian Mathematical Surveys} \textbf{56} (2001), 803--916. 
%
\bibitem{Dai13}
Dai, J.G.\ and Miyazawa, M., 
Stationary distribution of a two-dimensional SRBM: geometric views and boundary measures, 
\textit{Queueing Systems} \textbf{74} (2013), 181--217. 
%
\bibitem{Fayolle95} 
Fayolle, G., Malyshev, V.A.\ and Menshikov, M.V.,
{\it Topics in the Constructive Theory of Countable Markov Chains},  
Cambridge University Press, Cambridge (1995).
%
\bibitem{Flajolet09}
Flajolet, P.\ and Sedgewick, R.,
\textit{Analytic Combinatorics}, 
Cambridge University Press, Cambridge (2009).
%
%
\bibitem{Foley05}
Foley, R.D.\ and McDonald, D.R., 
Large deviations of a modified Jackson network: Stability and rough asymptotics, 
\textit{The Annals of Applied Probability} \textbf{15(1B)} (2005), 519--541.
%
\bibitem{Horn91}
Horn, R.A.\ and Johnson, C.R.,  
\textit{Topics in Matrix Analysis}, 
Cambridge University Press, Cambridge (1991).
%
\bibitem{Keilson79}
Keilson, J.,  
\textit{Markov Chain Models -- Rarity and Exponentiality}, 
Springer-Verlag, New York (1979).
%
%
%
\bibitem{Kobayashi13}
Kobayashi, M.\ and Miyazawa, M., 
Revisit to the tail asymptotics of the double QBD process: Refinement and complete solutions for the coordinate and diagonal directions, 
\textit{Matrix-Analytic Methods in Stochastic Models} (2013), 145-185.
%
%
\bibitem{Latouche99} 
Latouche, G.\ and Ramaswami, V., 
{\it Introduction to Matrix Analytic Methods in Stochastic Modeling}, 
SIAM, Philadelphia (1999).
%
%
\bibitem{Markushevich05}
Markushevich, A.I., 
\textit{Theory of Functions of a Complex Variable},  
AMS Chelsea Publishing, Providence (2005). 
%
%
\bibitem{Miyazawa09}
Miyazawa, M., 
Tail decay rates in double QBD processes and related reflected random walks, 
\textit{Mathematics of Operations Research} \textbf{34(3)} (2009), 547--575.
%
\bibitem{Miyazawa11}
Miyazawa, M., 
Light tail asymptotics in multidimensional reflecting processes for queueing networks,  
\textit{TOP} \textbf{19(2)} (2011), 233--299.
%
%
\bibitem{Miyazawa15}
Miyazawa, M., 
Superharmonic vector for a nonnegative matrix with QBD block structure and its application to a Markov modulated two dimensional reflecting process, 
\textit{Queueing Systems} \textbf{81} (2015), 1--48.
%
\bibitem{Miyazawa21}
Miyazawa, M., 
Markov modulated fluid network process: Tail asymptotics of the stationary distribution, 
\textit{Stochastic Models} \textbf{37} (2021), 127--167. 
%
\bibitem{Neuts94}
Neuts, M.F.,  
\textit{Matrix-Geometric Solutions in Stochastic Models}, 
Dover Publications, New York (1994).
%
\bibitem{Neuts89}
Neuts, M.F., 
{\it Structured stochastic matrices of M/G/1 type and their applications}, 
Marcel Dekker, New York (1989).
%
\bibitem{Ney87}
Ney, P.\ and Nummelin, E., 
Markov additive processes I. Eigenvalue properties and limit theorems, 
\textit{The Annals of Probability} \textbf{15(2)} (1987), 561--592. 
%
\bibitem{Nummelin84}
Nummelin, E., 
\textit{General Irreducible Markov Chains and Non-negative Operators},
Cambridge University Press, Cambridge (1984).
%
\bibitem{Ozawa13}
Ozawa, T.,
Asymptotics for the stationary distribution in a discrete-time two-dimensional quasi-birth-and-death process, 
\textit{Queueing Systems} \textbf{74} (2013), 109--149.
%
\bibitem{Ozawa18}
Ozawa, T.\ and Kobayashi, M., 
Exact asymptotic formulae of the stationary distribution of a discrete-time two-dimensional QBD process, 
\textit{Queueing Systems}  \textbf{90} (2018), 351-403. 
%
\bibitem{Ozawa18b}
Ozawa, T.\ and Kobayashi, M.: 
Exact asymptotic formulae of the stationary distribution of a discrete-time 2d-QBD process: an example and additional proofs, 
arXiv:1805.04802 (2018).
%
\bibitem{Ozawa19}
Ozawa, T., 
Stability condition of a two-dimensional QBD process and its application to estimation of efficiency for two-queue models, 
\textit{Performance Evaluation} \textbf{130} (2019), 101--118.
%
\bibitem{Ozawa21}
Ozawa, T., 
Asymptotic properties of the occupation measure in a multidimensional skip-free Markov modulated random walk,
\textit{Queueing Systems}  \textbf{97} (2021), 125--161. 
%
%
\end{thebibliography}
\end{document}